\documentclass[11pt]{amsproc}
\usepackage{mathrsfs}
\usepackage{stmaryrd}
\usepackage{cases}
\usepackage{amssymb}
\usepackage{amsmath}
\usepackage{amsfonts}
\usepackage{graphicx}
\usepackage{amsmath,amstext,amsbsy,amssymb, color}
\usepackage{cancel}

\newtheorem{theorem}{Theorem}[section]
\newtheorem{lemma}[theorem]{Lemma}
\newtheorem{proposition}[theorem]{Proposition}
\newtheorem{corollary}[theorem]{Corollary}
\theoremstyle{definition}

\theoremstyle{remark}

\numberwithin{equation}{section} \errorcontextlines=0

\newcommand{\Pf}{\mathrm{Pf}}

\newcommand{\per}{\mathrm{per}}

\newcommand{\ot}{\otimes}

\newcommand{\si}{\sigma}
\newcommand{\GL}{\mathrm{GL}}

\newcommand{\sdet}{\mathrm{sdet}}

\begin{document}

\title[Minor identities for Sklyanin determinants]
{Minor identities for Sklyanin determinants}
\author{Naihuan Jing}
\address{Department of Mathematics, North Carolina State University, Raleigh, NC 27695, USA}
\email{jing@ncsu.edu}
\author{Jian Zhang}
\address{School of Mathematics and Statistics, Central China Normal University, Wuhan, Hubei 430079, China}
\email{jzhang@ccnu.edu.cn}
\thanks{{\scriptsize
\hskip -0.6 true cm MSC (2020): Primary: 20G42; Secondary: 15B33, 15A75, 58A17, 15A15, 13A50
}}

\thanks{{\scriptsize\hskip -0.6 true cm *Corresponding author}}

\begin{abstract} We explore the invariant theory of quantum symmetric spaces of orthogonal and symplectic types by employing R-matrix techniques. Our focus involves establishing connections among the quantum determinant, Sklyanin determinants associated with the orthogonal and symplectic cases, and the quantum Pfaffians over the symplectic quantum space. Drawing inspiration from twisted Yangians, we not only demonstrate but also extend the applicability of q-Jacobi identities, q-Cayley's complementary identities, q-Sylvester identities, and Muir's theorem to Sklyanin minors in both orthogonal and symplectic types, along with q-Pfaffian analogs in the symplectic scenario. Furthermore, we present expressions for Sklyanin determinants and quantum Pfaffians in terms of quasideterminants.
\end{abstract}
\maketitle
\section{Introduction}

One of the central aspects of representation theory revolves around the general linear group and the symmetric group, establishing a close
connection with classical invariant theory \cite{W, H, GW}. Representations of the general linear group can be mostly generalized to the other classical groups (the orthogonal and symplectic groups). All these classical groups and their representations
have profound applications in various contexts, especially with several classical
combinatorial identities such as Capelli identities, Sylvester identities, etc, for more background see \cite{W}. For instance, the irreducible characters of the general linear group are recognized as the Schur symmetric polynomials indexed by the Young diagrams corresponding to the highest weights, similarly the
irreducible characters of the orthogonal and symplectic groups are also realized as the orthogonal and symplectic Schur functions \cite{Lit, Mac}, and many
of the classical identities are expressed in terms of determinants, Pfaffians and the like.

The Yangian algebra $Y(\mathfrak{g})$ was introduced by Drinfeld \cite{Dr} to solve the quantum Yang-Baxter equation over the finite-dimensional simple Lie algebra $\mathfrak g$. It is known that the Yangian algebra $Y(\mathfrak{gl}_n)$ is closely related to the classical representation theory of the general linear group and invariant theory \cite{M}. The quantum determinant of $Y(\mathfrak{gl}_n)$,
called {\it the Yangian determinant} in this paper, has enjoyed similar properties to the usual determinant, especially the Capelli identity associated with the Yangian determinant
gives rise to a complete set of generators for the center $ZY(\mathfrak{gl}_n)$.
This and other combinatorial properties of the Yangian determinant have been well studied in \cite{M}, where one can also see quite some classical identities have been generalized to the Yangian $Y(\mathfrak{gl}_n)$ and its determinant.

The twisted Yangians $Y^{\pm}(n)$ of Olshanski \cite{O} are certain subalgebras of $Y(\mathfrak{gl}_n)$ corresponding to the orthogonal Lie algebra
$\mathfrak o_n$ and the symplectic Lie algebra $\mathfrak {sp}_n$, which also provide contexts for generalized combinatorial identities \cite{N} associated with their quantum determinants: the Sklyanian determinant.
Again the coefficients of the Sklyanian determinant $\sdet(Y^{\pm}(u))$ generate the center $ZY^{\pm}(n)$ \cite{M, M2}.

The quantum group $\mathrm{GL}_q(n)$ was introduced by Faddeev-Reshetikhin-Tacktajan \cite{FRT} as a quadratic algebra generated by the $t_{ij} \, (1\leq i, j\leq n)$ defined by the $RTT$ relation for $T=(t_{ij})$
based on the trigonometric R-matrix, and the quantum determinant $\det_q(T)$ also generates the center of
$\mathrm{GL}_q(n)$ (see also \cite{DD, BG, RTF, LS, NUW}). In \cite{JY} and \cite{JR} we have studied the quantum symmetric spaces corresponding to the orthogonal and symplectic group as certain co-ideals of $\mathrm{GL}_q(n)$, and we have shown that the quantum Pfaffian is a special central element in the quantum symmetric space of symplectic type.
In \cite{No}
Noumi studied spherical functions on the quantum symmetric spaces.
He showed that the quantum determinant can be regarded as a quantum Pfaffian in the symplectic case
and raised a question 
on how to represent the square of the quantum determinant $\det_q(T)^2$ inside the
 quantum symmetric space of orthogonal type.

This paper aims to study a similar quantum invariant theory for the quantum symmetric spaces of orthogonal and symplectic types.
In the first part of the paper, we formulate the quantum symmetric spaces using the R-matrix subject to certain reflective RTT equations and
introduce the Sklyanin determinant in both cases. We show that the quantum symmetric spaces are characterized by the reflective RTT equations,
which are very much analogous to twisted Yangian algebras \cite{M}. In particular,
we introduce the Sklyanin determinant of the matrix $X$ and show that it generates the center of
the quantum symmetric space of orthogonal type. In the symplectic case, the center is generated by quantum Pfaffian. Moreover, we will show that $\det_q(T)^2$ is exactly the Sklyanin determinant $\sdet_{q}(X)$ (up to a constant) as a special element in the quantum symmetric space $A_q(X_N)$, thus
answering the aforementioned question of Noumi.

In the second part of the paper we give several identities for the quantum Sklyanin determinant $\sdet_q (X)$ and the associated quantum Pfaffian
$\mathrm{Pf}_q(X)$. The key identities can be expressed as
minor identities for the quantum Sklyanin determinant, which correspond to the classical identities for the quantum determinant over the quantum general linear group.
We will generalize the $q$-Jacobi identities, $q$-Cayley's complementary identities, the $q$-Sylvester identities and Muir's theorem to Sklyanin determinants
both in the quantum orthogonal and quantum symplectic situations.
In a sense, we have generalized several key identities for the general linear, orthogonal, and symplectic groups to their counterparts for the quantum general linear group and quantum symmetric spaces in the orthogonal and symplectic types.

In the case of quantum general linear semigroup $A_q(Mat_N)$, Krob and Leclerc \cite{KL} have expressed the quantum determinant $\det_q(T)$
 as a product of successive commuting (principal) quasideterminants defined by
Gelfand and Retakh \cite{GR}. At the end of the paper, we extend Krob-Leclerc's formula to derive similar identities for the Sklyanin determinants and quantum Pfaffians in terms of quasideterminants associated with the generator matrix $X$ of the quantum symmetric spaces.

\section{The quantum coordinate algebras} 

Let $R$ be the matrix in $\mathrm{End}(\mathbb C^N\otimes \mathbb C^N)\simeq\mathrm{End}(\mathbb C^N)^{\ot 2}$:
\begin{equation}
R=q\sum_{i=1}^Ne_{ii}\otimes e_{ii}+ \sum_{i\neq j}e_{ii}\otimes e_{jj}
+(q-q^{-1})\sum_{i>j}e_{ij}\otimes e_{ji},
\end{equation}
where $e_{ij}$ are the unit matrices in $\mathrm{End}(\mathbb C^N)$. It is known that
$R$ satisfies the well-known Yang-Baxter equation:
\begin{align*}
R_{12}R_{13}R_{23}=R_{23}R_{13}R_{12},
\end{align*}
where $R_{ij}\in \mathrm{End}(\mathbb C^N\otimes \mathbb C^N\otimes \mathbb C^N)$ acts on the $i$th and $j$th copies of $\mathbb C^N$ as $R$ does on $\mathbb C^N\otimes \mathbb C^N$.


Let $P$  be the permutation operator on $\mathbb{C}^N \ot \mathbb{C}^N$ defined by $P(w\ot v)=v\ot w$, $w, v\in\mathbb C^N$.
We define two $R$-matrices $R^{\pm}$ associated with $R$ by $R^+=PRP$, $R^-=R^{-1}$, then 
\begin{align}\nonumber
&R^+=q\sum_{i}e_{ii}\otimes e_{ii}+ \sum_{i\neq j}e_{ii}\otimes e_{jj}
+(q-q^{-1})\sum_{i<j}e_{ij}\otimes e_{ji}, \\ \nonumber
&R^-=q^{-1}\sum_{i}e_{ii}\otimes e_{ii}+ \sum_{i\neq j}e_{ii}\otimes e_{jj}
-(q-q^{-1})\sum_{i>j}e_{ij}\otimes e_{ji}.
\end{align}

Viewing $q$ as a variable, let $\mathbb C(q)$ be the field of
rational functions in $q$.
Following \cite{FRT}, we introduce the quantum coordinate algebra  $A_q(Mat_N)$ of  generic matrices
as the unital associative algebra generated by $t_{ij}, 1 \leq i,j \leq N$
subject to the quadratic relations defined by the matrix equation
\begin{align}\label{RTT}
R T_{1} T_{2}=T_{2}T_{1} R
\end{align}
in $\mathrm{End}(\mathbb C^N\ot \mathbb C^N)\ot \mathrm{M}_{q}(N)$,
where $T=(t_{ij})$, $T_1=T\ot I$, and $T_2=I\ot T$.

The quadratic defining relations are explicitly written as follows:
\begin{align}\label{relation1}
&t_{ik}t_{il}=qt_{il}t_{ik}, \\ \label{relation2}
&t_{ik}t_{jk}=qt_{jk}t_{ik}, \\ \label{relation3}
&t_{il}t_{jk}=t_{jk}t_{il}, \\ \label{relation4}
&t_{ik}t_{jl}-t_{jl}t_{ik}=(q-q^{-1})t_{il}t_{jk},
\end{align}
where $i<j$ and $k<l$.

The
algebra  $A_q(Mat_N)$ is a bialgebra under the comultiplication
 $A_q(Mat_N)\longrightarrow A_q(Mat_N)\ot A_q(Mat_N)$
 defined by
\begin{equation}
\Delta(t_{ij})=\sum_{k=1}^{N} t_{ik}\otimes t_{kj},
\end{equation}
and the counit given by $\varepsilon(t_{ij})=\delta_{ij}$. We will briefly write the coproduct
as $\Delta(T)=T\dot{\ot}T$.

Let $I$ and $J$ be two (ordered) subsets of $\{1,2,\cdots,N\}$ with identical cardinality $r$: 
$i_1<i_2<\cdots<i_r\in I$ and $j_1<j_2<\cdots<j_r\in J$. The quantum $r$-minor is defined as \cite{NYM, HH}
\begin{equation}
\begin{split}
\xi^{i_1,\cdots,i_r}_{j_1,\cdots,j_r}
&=\sum_{\sigma\in S_r}(-q)^{l(\sigma)}t_{i_1,j_{\sigma(1)}}\cdots t_{i_r,j_{\sigma(r)}}\\
&=\sum_{\sigma\in S_r}(-q)^{l(\sigma)}t_{i_{\sigma(1)},j_1}\cdots t_{i_{\sigma(r)},j_r},
\end{split}
\end{equation}
where $l(\sigma)=|\{(i, j)|i<j, \si_i>\si_j\}|$ is the inversion number of $\sigma$.
The second equality follows from relation \eqref{relation3}. In particular,
the quantum determinant of $T$ is the $n$-minor
\begin{equation}
\begin{split}
{\det}_q(T)=
\xi^{1,\cdots,N}_{1,\cdots,N}.
\end{split}
\end{equation}
The center of $A_q(Mat_N)$ is generated by ${\det}_q(T)$ and $\Delta({\det}_q(T))={\det}_q(T)\otimes {\det}_q(T)$ (cf. \cite{JZ2}).
The coordinate ring $\mathrm{GL}_{q}(N)$ is defined by adjoining the inverse of the quantum
determinant $\det_q(T)$ to
 $A_q(Mat_N)$.
It has a Hopf algebra structure with the
antipode $S$ of $\mathrm{GL}_{q}(N)$ given by the anti-automorphism such that
\begin{equation}
TS(T)=S(T)T=I,
\end{equation}
where $S(T)=S(t_{ij})_{1\leq i,j\leq N}$.

The quantum coordinate algebra $A_q(X_N)$ of symmetric (resp. anti-symmetric) matrices
is defined as a noncommutative algebra
generated by $x_{ij},1\leq i,j\leq N$ subject to the reflection relation \eqref{RXRX} and symmetry relation \eqref{q-sym} (resp. \eqref{q-antisym}):
\begin{equation}\label{RXRX}
 RX_1R^{t}X_2=X_2 R^{t}X_1R,
\end{equation}
where $X=(x_{ij})_{N\times N}$ and $R^t=R^{t_1}$ denotes the partial transpose in the first tensor factor; plus the relations
\begin{align}\label{q-sym}
 &\mathrm{Case (O)}: \qquad x_{ij}=qx_{ji};\\  \label{q-antisym}
 &\mathrm{Case (Sp)}: \qquad x_{kk}=0,  \quad x_{ji}=-qx_{ij}.
 \end{align}
where $1\leq i<j\leq N$, $k\in\{1, \cdots, N\}$. In the symplectic case, $N$ is even.
 Using
\begin{equation}
R^{t_1}=q\sum_{1\leq i\leq N}e_{ii}\otimes e_{ii}+ \sum_{i\neq j}e_{ii}\otimes e_{jj}
+(q-q^{-1})\sum_{i>j}e_{ji}\otimes e_{ji}.
\end{equation}
the reflection relations are explicitly written as
\begin{align}\notag
&q^{\delta_{jk}+\delta_{ij}}x_{ik}x_{jl}
-q^{\delta_{kl}+\delta_{il}}x_{jl}x_{ik}\\
=&(q-q^{-1})(\delta_{i<l}q^{\delta_{kl}}x_{ji}x_{lk}
-\delta_{j<k}q^{\delta_{ij}}x_{ij}x_{kl})\\ \notag
&\qquad +(q-q^{-1})q^{\delta_{ik}}(\delta_{k<l}-\delta_{j<i})x_{jk}x_{il}\\ \notag
&\qquad +(q-q^{-1})^2(\delta_{i<k<l}-\delta_{j<i<k})x_{ji}x_{kl}
\end{align}
where $\delta_{i<j}$ or $\delta_{i<j<k}$ equals $1$ if the subindex inequality is satisfied and $0$ otherwise.
In the orthogonal case, the relations  \eqref{RXRX} and \eqref{q-sym} can be written as:
\begin{align}\label{q symmetric}
&x_{ij}=qx_{ji}, \qquad i<j, \\
&x_{ik}x_{jk}=q x_{jk} x_{ik}, \qquad i<j<k, \\
&x_{ik}x_{il}=q x_{il} x_{ik}, \qquad i<k<l, \\
&x_{ij}x_{jj}=q^2 x_{jj} x_{ij}, \qquad i<j, \\
&x_{ii}x_{ij}=q^2 x_{ij} x_{ii}, \qquad i<j, \\
&x_{ii}x_{jk}-x_{jk}x_{ii}=q^{-1}(q^2-q^{-2}) x_{ij} x_{ik}, \qquad i<j<k, \\
&x_{ij}x_{kk}-x_{kk}x_{ij}=q^{-1}(q^2-q^{-2}) x_{ik} x_{jk}, \qquad i<j<k, \\
&x_{ii}x_{jj}-x_{jj}x_{ii}=q^{-1}(q^2-q^{-2}) x_{ij}^2, \qquad i<j, \\
&x_{il}x_{jk}=x_{jk}x_{il} ,  \qquad i<j\leq k<l,\\
&x_{ij}x_{jk}-qx_{jk}x_{ij}=q(q-q^{-1}) x_{jj} x_{ik}, \qquad i<j<k, \\ \label{relation 11}
&x_{ik}x_{jl}-x_{jl}x_{ik}=(q-q^{-1})x_{il}x_{jk}, \qquad i<j<k<l,\\    \label{relation 12}
&x_{ij}x_{kl}-x_{kl}x_{ij}=(q-q^{-1})(x_{ik}x_{jl}+q^{-1}x_{il}x_{jk}), \qquad i<j<k<l.
\end{align}
Using relation \eqref{relation 11},
Eq. \eqref{relation 12} can be rewritten as
\begin{equation}
x_{ij}x_{kl}-x_{kl}x_{ij}= qx_{ik}x_{jl}-q^{-1}x_{jl}x_{ik}, \qquad i<j<k<l.
\end{equation}

In the symplectic case, the relations \eqref{RXRX} and \eqref{q-antisym} can be written as
\begin{align}\label{zero diagonal}
& x_{ii}=0, \\
& x_{ji}=-qx_{ij}, \qquad i<j, \\ \label {r q commute}
& x_{ik}x_{il}=q x_{il} x_{ik}, \qquad k<l, \\ \label {c q commute}
& x_{ik}x_{jk}=q x_{jk} x_{ik}, \qquad i<j, \\ \label{commute}
& x_{il}x_{jk}=x_{jk}x_{il} ,   \qquad i<j<k<l,\\ \label{relation noncommute 1}
& x_{ik}x_{jl}-x_{jl}x_{ik}=(q-q^{-1})x_{il}x_{jk}, \qquad i<j<k<l,\\ \label{relation noncommute}
& x_{ij}x_{kl}-x_{kl}x_{ij}=(q-q^{-1})(x_{ik}x_{jl}-qx_{il}x_{jk}), \qquad i<j<k<l.
\end{align}

Using relation \eqref{relation noncommute 1},
Eq. \eqref{relation noncommute} can be rewritten as
\begin{equation}
x_{ij}x_{kl}-x_{kl}x_{ij}=qx_{jl}x_{ik}-q^{-1}x_{ik}x_{jl},\qquad i<j<k<l.
\end{equation}

The following lemma follows from the explicit relations.
\begin{lemma}\label{lemma basis}
The monomials
\begin{equation}
\begin{aligned}
&\text{Case(O):}   &x^C=x_{11}^{c_{11}}x_{12}^{c_{12}}\cdots x_{1N}^{c_{1N}}
x_{22}^{c_{22}}x_{23}^{c_{23}}\cdots x_{2N}^{c_{2N}}
\cdots x_{NN}^{c_{NN}}\\
&\text{Case(Sp):}    &x^C=x_{12}^{c_{12}}x_{13}^{c_{13}}\cdots x_{1N}^{c_{1N}}
x_{23}^{c_{23}}x_{24}^{c_{24}}\cdots x_{2N}^{c_{2N}}
\cdots x_{N-1,N}^{c_{N-1,N}}
\end{aligned}
\end{equation}
span the algebra $A_q(X_N)$, where $C=(c_{ij})_{1\leq i,j\leq N}$ are (or strictly) upper triangular matrices with nonnegative integers.

\end{lemma}
We define the matrix $J(a)\in \mathrm{End}(\mathbb C^{N}\otimes \mathbb C^{N})$ by
\begin{equation}
\begin{aligned}
&\text{Case(O):}\quad J(a)=\sum_{i=1}^{N}a_ie_{ii},\\
&\text{Case(Sp):}\quad J(a)=\sum_{i=1}^{n}a_i(e_{2i-1,2i}-qe_{2i,2i-1}),
\end{aligned}
\end{equation}
where $a_i\in \mathbb C$ ($1\leq i\leq N$) are all nonzero numbers, and $n=N/2$ in the symplectic case.

\begin{theorem}\label{embedding thm}
The map $X\mapsto TJ(a)T^{t}$ is
a homomorphic embedding $\phi: A_q(X_N) \longrightarrow A_q(\mathrm{Mat}_N)$.
\end{theorem}

\begin{proof} Write $TJ(a)T^t=\tilde{X}=(\tilde{x}_{ij})$.
In the symplectic case,
$
\tilde{x}_{ii}=0$.
$\tilde{x}_{ij}=\sum_{k=1}^{N}{\det}_q(T_{2k-1,2k}^{ij})a_{k}$ and
$\tilde{x}_{ji}=-q\tilde{x}_{ij}$
for $i<j$.
In the orthogonal case,
$\tilde{x}_{ij}=\sum_{k=1}^{N}a_k t_{ik}t_{jk}$ and
$\tilde{x}_{ij}=q\tilde{x}_{ji}$ for $i<j$.

Further, we need to verify that
\begin{equation}
RT_1J_1(a)T_1^tR^{t}T_2J_2(a)T_2^t=T_2J_2(a)T_2^t R^{t}T_1J_1(a)T_1^tR.
\end{equation}

The relations $R^{t_1t_2}=R^+$ and  \eqref{RTT} imply that
\begin{align}
&T_1^tR^{t_1}T_2=T_2 R^{t_1} T_1^t,\\
&RT_1^tT_2^t=T_2^tT_1^t R,\\
&T_1R^{t_1}T_2^t=T_2^t R^{t_1} T_1.
\end{align}

By direct computation, one has that
\begin{align}
R J_1(a)
 R^{t}J_2(a)=J_2(a) R^{t}J_1(a) R.
\end{align}
Consequently
\begin{equation}
\begin{split}
&RT_1J_1(a)T_1^tR^{t}T_2 J_2(a)T_2^t
\\
&=T_2T_1 R J_1(a)  R^{t} J_2(a) T_1^t T_2^t
\\
&=T_2 T_1 J_2(a) R^{t}  J_1(a) R T_1^t T_2^t\\
&=T_2J_2(a) T_2^t R^{t}T_1J_1(a)T_1^tR.
\end{split}
\end{equation}

This proves that $\phi$ is an algebra homomorphism.
We now check that the images of the monomials in Lemma
\ref{lemma basis} are
linearly independent under $\phi$ .

Let $A_q(\mathrm{Mat}_N)'$ be the $\mathbb C[q,q^{-1}]$-subalgebra of $A_q(\mathrm{Mat}_N)$
generated by the elements $t_{i,j}$,$1\leq i,j\leq N$.
Then there exists an isomorphism
\begin{equation}\label{specialize}
A_q(\mathrm{Mat}_N)'\otimes_{\mathbb C[q,q^{-1}]}\mathbb C \cong A(\mathrm{Mat}_N)
\end{equation}
with the action of ${\mathbb C[q,q^{-1}]}$ on $\mathbb C$ defined via the evaluation $q=1$.
Suppose  there is
a nontrivial linear relation among ordered monomials in the $x^C$:
\begin{equation}
\sum_{C} a_Cx^C=0,
\end{equation}
where $a_C\in \mathbb C [q,q^{-1}]$. We can assume that at least one coefficient $a_C$ does not vanish at $q=1$.
Take the image of \eqref{specialize}, a nontrivial linear combination of the image of $\phi(x^C)$ in $A_q(\mathrm{Mat}_N)$ equal to zero. This is a contradiction.
\end{proof}

The following proposition follows from the proof of the Theorem \ref{embedding thm}.
\begin{proposition}
The monomials
\begin{equation}
\begin{aligned}
&\text{Case(O):}   &x^C=x_{11}^{c_{11}}x_{12}^{c_{12}}\cdots x_{1N}^{c_{1N}}
x_{22}^{c_{22}}x_{23}^{c_{23}}\cdots x_{2N}^{c_{2N}}
\cdots x_{NN}^{c_{NN}}\\
&\text{Case(Sp):}    &x^C=x_{12}^{c_{12}}x_{13}^{c_{13}}\cdots x_{1N}^{c_{1N}}
x_{23}^{c_{23}}x_{24}^{c_{24}}\cdots x_{2N}^{c_{2N}}
\cdots x_{N-1,N}^{c_{N-1,N}}
\end{aligned}
\end{equation}
form a basis of the algebra $A_q(X_N)$, where $C=(c_{ij})_{1\leq i,j\leq N}$ runs through (or strictly) upper triangular matrices with nonnegative integers.
\end{proposition}

Moreover, one has that
$\Delta(\tilde{x}_{ij})\in A_q(Mat_N)\otimes A_q(X)$.
Explicitly,
\begin{equation}
\begin{split}
&\text{Case(O):} \quad \Delta(\tilde{x}_{ij})=\sum_{r,s=1}^N t_{ir}t_{js}\otimes \tilde{x}_{rs},\\
&\text{Case(Sp):} \quad \Delta(\tilde{x}_{ij})=\sum_{r<s}{\det}_q(T_{r,s}^{ij})\otimes \tilde{x}_{rs}.
\end{split}
\end{equation}
Therefore, the algebras $A_q(X)$ can be regarded as the left coideal subalgebras of $A_q(Mat_N)$ via the embedding.
Similarly, the map $X\mapsto T^tJ(a)T$ also defines a homomorphic embedding, and the image of $A_q(X_N)$ is a right coideal subalgebra of $A_q(Mat_N)$.

We remark that the quantum coordinate algebras $A_q(X_N)$ were studied in \cite{JY, JR} as certain invariant subalgebras of $A_q(Mat_N)$ annihilated by
$q$-differential operators and they are quadratic quantum algebras in the sense of Manin \cite{M-a}. Their dual pictures are important examples of quantum symmetric pairs studied in general \cite{L} 
associated to the quantum enveloping algebra \cite{Dr, J}.

\section{The Sklyanin determinant}
We introduce the spectral-dependent $R$-matrix
$R(\lambda)=\lambda R^+ -{\lambda} ^{-1}R^{-}$,
which satisfies the Yang-Baxter equation:
\begin{equation}\label{YBE}
R_{12}(\lambda/\mu)R_{13}(\lambda)R_{23}(\mu)=R_{23}(\mu)R_{13}(\lambda)R_{12}(\lambda/\mu).
\end{equation}

Let $\hat{R}(\lambda)=R(\lambda)P$, then the Yang-Baxter equation is equivalent
to the braid relation 
\begin{equation}\label{YBE1}
\hat{R}_{12}(\lambda/\mu)\hat{R}_{23}(\lambda)\hat{R}_{12}(\mu)
=\hat{R}_{23}(\mu)\hat{R}_{12}(\lambda)\hat{R}_{23}(\lambda/\mu),
\end{equation}
and relation \eqref{RXRX} is then equivalent to
\begin{equation}\label{RBRB1}
\hat R(\lambda)X_1R^{t}X_2=X_1R^{t}X_2 \hat R(\lambda).
\end{equation}

In this paper, we use
the $v$-based quantum number $
[n]_v=1+v+\cdots+v^{n-1}$
and the quantum factorial $[n]_v!=[1]_v[2]_v\cdots [n]_v$ for any natural number $n\in\mathbb N$. In particular, $[0]!=1$.
Let $A_r$ be the $q$-antisymmetrizer:
\begin{equation}
A_{m}=\frac{1}{[m]_{q^2}!}\sum_{c_1<c_2<\cdots<c_m,\sigma,\tau \in S_m}(-q)^{l(\sigma)+l(\tau)}
e_{c_{\sigma(1)}c_{\tau(1)}}\ot \cdots \ot e_{c_{\sigma(m)}c_{\tau(m)}},
\end{equation}
and it is an idempotent: $A_{r}^2=A_r$.
It  follows from the Yang-Baxter equation that
\begin{equation}
A_{2}=\frac{1}{q^2-q^{-2}}\hat{R} (q^{-1}),\   A_{m+1}=\frac{1}{q^{m+1}-q^{-m-1}}A_{m } \hat{R}_{m,m+1} (q^{-m})A_{m }.
\end{equation}

For any permutation $i_1,i_2,\dots,i_m$ of $1,2,\dots,m$ we denote
\begin{equation}
<X_{i_1}X_{i_2}\cdots X_{i_m}>
=X_{i_1} (R^{t}_{i_1i_2}\cdots R^{t}_{i_1i_m})
X_{i_2} (R^{t}_{i_2 i_3}\cdots R^{t}_{i_2i_m})
\cdots X_{i_m}
\end{equation}

The following proposition follows from the reflection relation \eqref{RBRB1}
and the variant form of the Yang-Baxter equation \eqref{YBE}:
\begin{equation}\label{variant YBE}
\begin{split}
\hat R(\lambda)_{ij}R_{ik}^tR_{jk}^t=R_{ik}^tR_{jk}^t \hat R(\lambda)_{ij},\\
\hat R(\lambda)_{jk}R_{ij}^tR_{ik}^t=R_{ij}^tR_{ik}^t \hat R(\lambda)_{jk}.
\end{split}
\end{equation}
\begin{proposition} One has that
\begin{equation}
A_m \langle X_1,\dots,X_m \rangle =\langle X_1,\dots,X_m \rangle A_m.
\end{equation}
\end{proposition}

The element ${[m]_{q^2}!}A_m \langle X_1,\dots,X_m \rangle$ can be written as
\begin{equation}
\sum X^{i_1\cdots i_m}_{j_1\cdots j_m}\otimes e_{i_1 j_1}\otimes \dots \otimes e_{i_mj_m},
\end{equation}
where the sum is taken over all $i_k,j_k\in \{1,2,\cdots, N\}$.
We call $X^{i_1\cdots i_m}_{j_1\cdots j_m}$
the {\it Sklyanin minor} associated to the rows $i_1\cdots i_m$ and the columns  $j_1\cdots j_m$.

Clearly if $i_k=i_l$ or $j_k=j_l$ for $k\neq l$, then $X^{i_1\cdots i_m}_{j_1\cdots j_m}=0$.
Suppose that $i_1<i_2<\dots<i_m$ and $j_1<j_2<\dots<j_m$. Then the Sklyanin minors
satisfy the relations:
\begin{equation}
\begin{split}
X^{i_{\sigma(1)}\cdots i_{\sigma(m)}}_{j_{\tau(1)}\cdots j_{\tau(m)}}
=(-q)^{l(\sigma)+l(\tau)}X^{i_1\cdots i_m}_{j_1\cdots j_m}.
\end{split}
\end{equation}
Note that $X^i_j=x_{ij}$. In particular,
$X^{1\cdots N}_{1\cdots N}$ is called the {\it Sklyanin determinant} and denoted by $\sdet_q(X)$. Sometimes we will simply write
$\sdet(X)$ if the parameter $q$ is clear from the context.

\begin{proposition}\label{sdet}
Regarding $A_q(X_N)$ as subalgebra of $A_q(\mathrm{Mat}_N)$ through the embedding $\phi$ in Thm. \ref{embedding thm}, we have that
\begin{equation}
\sdet_{q} (X)=\gamma_{N,J(a)}{\det}_q(T) ^2,
\end{equation}
where $\gamma_{N,J(a)}=a_1a_2\cdots a_N$ for case (O) and $q^{3n}(a_1a_2\cdots a_n)^2$ for case (Sp).
\end{proposition}

\begin{proof} The image of $A_N \langle X_1,\dots,X_N \rangle=\sdet_q(X) A_N$ under $\phi$ takes the following form
\begin{equation}
\begin{split}
A_NT_1J_1(a)T_1^t (R^{t}_{12}\cdots R^{t}_{1N})
T_2J_2(a)T_2^t (R^{t}_{23}\cdots R^{t}_{2N})\cdots T_NJ_N(a)T_N^t
\end{split}
\end{equation}
Applying the relation $T_i^t R_{ij}^tT_j=T_j R_{ij}^t T_i^t$, we get that

\begin{equation}
\sdet_{q}(X)A_N
=A_N T_1\cdots T_N \langle J_1(a),\dots,J_N(a) \rangle   T_1^t\cdots T_N^t
\end{equation}

By relation \eqref{RTT}, we have
\begin{equation}
RT_1^tT_2^t=T_2^tT_1^tR.
\end{equation}
Therefore,
\begin{equation}
A_N T_1^t\cdots T_N^t ={\det}_{q}(T)A_N,
\end{equation}
and
\begin{equation}
\sdet_{q}(X)A_N
=({\det}_{q}(T))^2 A_N \langle J_1(a),\dots,J_N(a) \rangle.
\end{equation}

In the orthogonal case, it is easy to see that
\begin{equation}
A_N \langle J_1(a),\dots,J_N(a) \rangle=a_1\cdots a_N A_N.
\end{equation}

Now we consider symplectic case.
Let $J=\sum_{i=1}^{n}e_{2i-1,2i}-qe_{2i,2i-1}$,
and $U=\sum_{i=1}^{n} u_i(e_{2i-1,2i-1}+e_{2i,2i}),$
where $u_i^2=a_i$. Then $\det(U)=a_1\cdots a_n$, $J(a)=UJU^t$
and
\begin{equation}
A_N \langle J_1(a),\dots,J_N(a) \rangle
=\det(U)^2 A_N \langle J_1,\dots,J_N \rangle .
\end{equation}

Apply $A_N \langle J_1,\dots,J_N \rangle $ to the basis vector
\begin{equation}
v=e_{2n-1}\otimes e_{2n-3}\otimes\cdots \otimes e_1 \otimes e_2\otimes e_4\otimes \cdots \otimes e_{2n},
\end{equation}
which can be written as
\begin{equation}\label{action w}
\begin{split}
A_NJ_1 (R^{t}_{12}\cdots R^{t}_{1N})
J_2 (R^{t}_{23}\cdots R^{t}_{2N})\cdots J_n (R^{t}_{n,n+1}\cdots R^{t}_{n,N})w
\end{split}
\end{equation}
where
\begin{equation}
w=e_{2n-1}\otimes e_{2n-3}\otimes\cdots \otimes e_1 \otimes e_1\otimes e_3\otimes \cdots e_{2n-1}.
\end{equation}

Let $A_{N-i}'$ be the $q$-antisymmetrizer on the indices $\{i+1,\dots,N\}$.
Then \eqref{action w} can also be rewritten as
\begin{equation*}
\begin{split}
A_NJ_1 (R^{t}_{12}\cdots R^{t}_{1N})A_{N-1}'
J_2 (R^{t}_{23}\cdots R^{t}_{2N})A_{N-2}'\cdots A_{n+1}'J_n (R^{t}_{n,n+1}\cdots R^{t}_{n,N})w.
\end{split}
\end{equation*}
Since
\begin{equation}
R_{ij}^t e_k\otimes e_k= q e_k\otimes e_k+(q-q^{-1})\sum_{l<k}e_l\otimes e_l,
\end{equation}
we conclude that
\begin{equation}
A_N \langle J_1,\dots,J_N \rangle v=
(-q^2)^n A_N w'
\end{equation}
where
\begin{equation}
w'=e_{2n}\otimes e_{2n-2}\otimes\cdots \otimes e_2 \otimes e_1\otimes e_3\otimes \cdots \otimes e_{2n-1}.
\end{equation}
Since $A_Nw'=(-q)^nA_Nv$, we have that
$
A_N \langle J_1,\dots,J_N \rangle v= q^{3n} A_N v
$. Therefore,
\begin{equation}
\gamma_{N,J(a)}=
q^{3n}(a_1a_2\cdots a_n)^2.
\end{equation}
\end{proof}

We remarked that the proposition solved the question raised in \cite[Rem. 4.12]{No}.
Since the quantum determinant $\det_q(T)$ can be regarded as a quantum Pfaffian \cite{No}, the proposition
confirms that $\Pf_q(X)^2=\sdet_q(X)$ up to a constant, which also solves the
puzzle between $\det_q(T)$ and $\Pf_q(X)$ \cite{JZ2}.

\begin{corollary}\label{sdet center}
The Sklyanin determinant $\sdet_{q}(X)$ belongs to the center of $A_q(X_N)$.
\end{corollary}
\begin{proof}
It follows from Theorem \ref{embedding thm} and Theorem \ref{sdet}.
\end{proof}

\begin{proposition}
Let $I=\{i_1,\cdots,i_m\}$ and $J=\{j_1,\cdots,j_m\}$
be two subsets of $\{1,2,\cdots,N\}$ such that $1\leq i_1<i_2<\dots <i_m\leq N$ and $1\leq j_1<j_2<\dots <j_m\leq N$.
Suppose that $a,b\in I\cap J$, then
\begin{equation}
 \begin{split}
x_{ab}X^I_J=X^I_Jx_{ab}.
 \end{split}
\end{equation}
\end{proposition}
\begin{proof}
  Let $A_{m}$ be the $q$-antisymmetrizer on the indices $\{1,\dots,m\}$.
It follows from the Yang-Baxter equation and the reflection relation that
  \begin{equation}
    \begin{split}
 R_{0m}\dots R_{01}X_0 R^t_{01}\dots R^t_{0m}A_m\langle X_1,\dots,X_m \rangle
 \\=
 A_m\langle X_1,\dots,X_m \rangle R^t_{01}\dots R^t_{0m} X_0 R_{0m}\dots R_{01}.
    \end{split}
   \end{equation}
By applying both sides to the vector $e_b\otimes e_{j_1}\otimes \dots \otimes e_{j_m}$ and  comparing the coeffients of
$e_a\otimes e_{i_1}\otimes \dots \otimes e_{i_m}$ we have that
\begin{equation}
  \begin{split}
    x_{ab}X^I_J=X^I_Jx_{ab}.
  \end{split}
 \end{equation}
\end{proof}

We define the auxiliary minor
$\check X ^{i_1,\dots,i_m}_{j_1,\dots,j_{m-1},c}$
by
\begin{equation}\begin{split}
&{[m]_{q^2}!}A_m \langle X_1,\dots,X_{m-1} \rangle
R^{t}_{1m}\cdots R^{t}_{m-1,m}\\
&=
\sum \check X^{i_1\cdots i_m}_{j_1\cdots j_{m-1},c}\otimes e_{i_1 j_1}\otimes \dots \otimes e_{i_m c},
\end{split}
\end{equation}
where the sum is taken over all $i_k,j_k,c\in \{1,2,\cdots, N\}$.

If $i_k=i_l$ or $j_k=j_l$ for $k\neq l$, then $\check X^{i_1\cdots i_m}_{j_1\cdots j_m}=0$.
Suppose that $i_1<i_2<\dots<i_m$ and  $j_1<j_2<\dots<j_{m-1}$. Then
\begin{equation}
\begin{split}
\check X^{i_{\sigma(1)}\cdots i_{\sigma(m)}}_{j_{\tau(1)}\cdots j_{\tau(m-1)},c}
=(-q)^{l(\sigma)+l(\tau)}\check X^{i_1\cdots i_m}_{j_1\cdots j_{m-1},c}.
\end{split}
\end{equation}
$\sigma\in S_m, \tau\in S_{m-1}$.

\begin{proposition}
Suppose that $i_1<i_2<\dots<i_m$ and  $j_1<j_2<\dots<j_{m-1}$. Then
\begin{equation}
X ^{i_1,\dots,i_m}_{j_1,\dots,j_m}
=\sum_{c=1}^{N}
\check X ^{i_1,\dots,i_m}_{j_1,\dots,j_{m-1},c}x_{c j_m}
\end{equation}
\end{proposition}
\begin{proof} The identity follows from the formula
\begin{equation}
A_m \langle X_1,\dots,X_m \rangle=A_m \langle X_1,\dots,X_{m-1} \rangle
R^{t}_{1m}\cdots R^{t}_{m-1,m}X_m.
\end{equation}
\end{proof}

\begin{proposition}\label{expansion auxiliary det}
Suppose that $i_1<i_2<\dots<i_{m-1}$, $j_2<\dots,j_{m-1}$,
$j_1\in \{i_1,\dots, i_m\}$ and $c\notin \{j_2,\dots,j_{m-1}\}$.
Then
\begin{equation}
\begin{split}
\check X^{i_1\cdots i_m}_{j_1\cdots j_{m-1},c}&=0, \qquad \text{ if }
c\notin \{i_1,\dots, i_m\},  \\
\check X^{i_1\cdots i_m}_{j_1\cdots j_{m-1},c}&=
\pm(-q)^{2l(I)}\sum_{r=1}^{m-1}(-q)^{r\mp 1}
 x_{i_r j_1}^t
 X^{i_1,\dots,\hat{i_r},\dots,i_{m-1}}_{j_2,\dots,j_{m-1}}
\end{split}
\end{equation}
if $c=i_m$, where $x^t_{ij}$ denote the $(i,j)$-th entry of the matrix $X^t$ and $l(I)$ is the
number of inversions of $i_1,i_2,\cdots,i_m$.

\end{proposition}
\begin{proof}
As $c\notin \{j_2,\dots,j_{m-1}\}$, we have that
\begin{equation}
R^{t}_{2m}\cdots R^{t}_{m-1,m}e_{j_2}\otimes \dots \otimes e_{j_{m-1}}\otimes e_c=e_{j_2}\otimes \dots \otimes e_{j_{m-1}}\otimes e_c
\end{equation}

Now let's compute
\begin{equation}\label{e:expand}
\begin{split}
A_m \langle X_1,\dots,X_{m-1} \rangle
R^{t}_{1m}\cdots R^{t}_{m-1,m}e_{j_1}\otimes \dots \otimes e_{j_{m-1}}\otimes e_c\\
=A_m X_1 R^{t}_{12}\cdots R^{t}_{1,m} \langle X_2,\dots,X_{m-1} \rangle
e_{j_1}\otimes \dots \otimes e_{j_{m-1}}\otimes e_c\\
\end{split}
\end{equation}

Let $A_{m-1}'$ be the $q$-antisymmetrizer numbered by indices $\{2,\dots,m\}$.
Then $A_m=A_mA_{m-1}'$, and we have the following relation:
\begin{equation}
A_{m-1}'R^{t}_{12}\cdots R^{t}_{1,m}=R^{t}_{12}\cdots R^{t}_{1m}A_{m-1}',
\end{equation}
which follows from the variant form of the Yang-Baxter equation \eqref{variant YBE}.
Then \eqref{e:expand} is continued to
\begin{equation}\label{expansion}
\begin{split}
&A_m X_1 R^{t}_{12}\cdots R^{t}_{1,m} A_{m}'
\langle X_2,\dots,X_{m-1} \rangle
e_{j_1}\otimes \dots \otimes e_{j_{m-1}}\otimes e_c\\
&=A_m X_1 R^{t}_{12}\cdots R^{t}_{1,m}
A_m'
\sum X^{k_2,\dots,k_{m-1}}_{j_2,\dots,j_{m-1}}
e_{j_1}\otimes e_{k_2}\otimes \dots \otimes e_{k_{m-1}}\otimes e_c
\end{split}
\end{equation}
where the sum runs through ${k_2<\dots<k_{m-1}}$.
Now let us divide it into the following cases:

(i) If $c\notin \{i_1,\dots, i_m\}$, then $c\neq j_1$ and
\begin{equation}
R_{1m}^t
e_{j_1}\otimes e_{k_2}\otimes \dots \otimes e_{k_{m-1}}\otimes e_c
=
e_{j_1}\otimes e_{k_2}\otimes \dots \otimes e_{k_{m-1}}\otimes e_c,
\end{equation}
So the basis vectors $e_{r_1}\otimes \dots\otimes e_{r_m}$ in the expansion of
\eqref{expansion}
only contain those with $c\in \{r_1,\dots, r_m\}$, thus
$\check X^{i_1\cdots i_m}_{j_1\cdots j_{m-1},c}=0$.

(ii) If $c=j_1=i_m$, then
\begin{equation}
A_m'
e_{j_1}\otimes e_{k_2}\otimes \dots \otimes e_{k_{m-1}}\otimes e_c=0
\end{equation}
if $k_r=j_1$ for some $2\leq r\leq m-1$.
Assume that $i_p<j_1<i_{p+1}$, then the coefficient of
$
e_{i_1} \otimes \dots  \otimes e_{i_m}
$
is
\begin{equation}
\begin{split}
&-(-q)^{2l(I)}\sum_{r=1}^{m-1}(-q)^{r}
 x_{i_r j_1}
 X^{i_1,\dots,\hat{i_r},\dots,i_{m-1}}_{j_2,\dots,j_{m-1}}\\
&+(-q)^{2l(I)}(q-q^{-1})\sum_{r=1}^{p}(-q)^{r}
 x_{j_1i_r }
X^{i_1,\dots,\hat{i_r},\dots,i_{m-1}}_{j_2,\dots,j_{m-1}}
 \end{split}
\end{equation}

(iii) Suppose $c=i_m$ and $j_1= i_{p}$ for some $1\leq p\leq m-1$. If $j_1\notin\{k_2,\dots,k_{m-1}\}$, then
 \begin{equation}
\begin{split}
A_m X_1 R^{t}_{12}\cdots R^{t}_{1,m}
X^{k_2,\dots,k_{m-1}}_{j_2,\dots,j_{m-1}}
e_{j_1}\otimes e_{k_{2}}\otimes \dots \otimes e_{k_{m-1}}\otimes e_c
\\
=A_m\sum_{k_1}x_{k_1j_1}X^{k_2,\dots,k_{m-1}}_{j_2,\dots,j_{m-1}}
e_{k_1}\otimes e_{k_2}\otimes \dots \otimes e_{k_{m-1}}\otimes e_c\\
\end{split}
\end{equation}

If $j_1=k_r$ for $2\leq r\leq m-1$, then
\begin{equation}
\begin{split}
&A_m X_1 R^{t}_{12}\cdots R^{t}_{1,m}
e_{j_1}\otimes e_{k_2}\otimes \dots \otimes e_{k_{m-1}}\otimes e_c\\
=&(-q)^{2-r}A_m X_1 R^{t}_{12}\cdots R^{t}_{1,m}
e_{j_1}\otimes e_{k_r}\otimes e_{k_2} \dots \hat{e_{k_r}}\dots e_{k_{m-1}}\otimes e_c\\
=&A_m \left(\sum_{k_1}qx_{k_1j_1}
e_{k_1}\otimes e_{k_2}\otimes \dots \otimes e_{k_{m-1}}\otimes e_c \right.\\
&\left.+(-q)^{2-r}(q-q^{-1})\sum_{k_1, j<j_1}x_{k_1j}
e_{k_1}\otimes e_{j}\otimes\dots \hat{e}_{k_r}\dots e_{k_{m-1}}\otimes e_c\right).
\end{split}
\end{equation}

The coefficient of
$
e_{i_1} \otimes \dots  \otimes e_{i_m}
$ in the expansion of
\eqref{expansion}
is
\begin{equation}
\begin{aligned}
&(-q)^{p-1}(-q)^{2l(I)}
x_{i_{p},j_1}X^{i_1,\dots,\hat{i_{p}},\dots,i_{m-1}}_{j_2,\dots,j_{m-1}}\\
&-(-q)^{2l(I)}\sum_{r\neq p}(-q)^{r}
 x_{i_r j_1}
 X^{i_1,\dots,\hat{i_r},\dots,i_{m-1}}_{j_2,\dots,j_{m-1}} \\
&+(q-q^{-1})(-q)^{2l(I)}
\sum_{r=1}^{p-1}(-q)^{r}
 x_{j_1i_r}
 X^{i_1,\dots,\hat{i_r},\dots,i_{m-1}}_{j_2,\dots,j_{m-1}}
 \end{aligned}
\end{equation}

In the cases (ii) and (iii), both coefficients of
$
e_{i_1} \otimes \dots  \otimes e_{i_m}
$  can be written as
\begin{equation}
\pm(-q)^{2l(I)}\sum_{r=1}^{m-1}(-q)^{r\mp 1}
 x_{i_r j_1}^t
X^{i_1,\dots,\hat{i_r},\dots,i_{m-1}}_{j_2,\dots,j_{m-1}}.
\end{equation}
\end{proof}

In order to produce an explicit formula for the Sklyanin determinant in the orthogonal case we introduce a map
\begin{center}
$\pi_{N}:S_{N}\rightarrow S_{N},\ p\mapsto p'$
\end{center}
which was used in the formula for the Sklyanin determinant for the twisted Yangians \cite{M, MRS}. The
map $\pi_N$ is defined inductively as follows. Given a set of positive integers $\omega_{1}<\cdots<\omega_{N}$, regard
 $S_{N}$ as the symmetric group of these indices. If $N=2$ we define $\pi_{2}$ as the identity map of $S_{2}\rightarrow S_{2}$.
 For $N>2$ define a map from the set of ordered pairs $(\omega_{k},\ \omega_{l})$ with $k\neq l$ into itself by the rule
\begin{equation}\label{permutation map}
\begin{split}
&(\omega_{k},\ \omega_{l})\mapsto(\omega_{l},\ \omega_{k})\ ,\ k,\ l<N,
\\
&(\omega_{k},\ \omega_{N})\mapsto(\omega_{N-1},\ \omega_{k})\ ,\ k<N-1,
\\
&(\omega_{N},\ \omega_{k})\mapsto(\omega_{k},\ \omega_{N-1})\ ,\ k<N-1,
\\
&(\omega_{N-1},\ \omega_{N})\mapsto(\omega_{N-1},\ \omega_{N-2})\ ,
\\
&(\omega_{N},\ \omega_{N-1})\mapsto(\omega_{N-1},\ \omega_{N-2})\ .
\\
\end{split}
\end{equation}
Let $p=(p_{1},\ \ldots,p_{N})$ be a permutation of the indices $\omega_{1}, \ldots, \omega_{N}$. Its image under the map $\pi_{N}$ is the permutation $p'=(p_{1}',\ \ldots,p_{N-1}',\ \omega_{N})$ , where the pair $(p_{1}',p_{N-1}')$ is the image of the ordered pair $(p_{1},p_{N})$ under the map \eqref{permutation map}. Then the pair $(p_{2}',p_{N-2}')$ is found as the image of $(p_{2},p_{N-1})$ under the map \eqref{permutation map} which is defined on the set of ordered pairs of elements obtained from $(\omega_{1},\ \ldots,\ \omega_{N})$ by deleting $p_{1}$ and $p_{N}$.
The procedure is completed in the same manner by determining consecutively
the pairs  $(p_{i}',p_{N-i}')$.

\begin{theorem}\label{sdet formula}
The Sklyanin determinant $\sdet_{q}(X)$ can be written explicitly as
\begin{equation}
\sdet_{q} (X)= \gamma_N\sum_{p\in S_N}(-q)^{l(p)-l(p')}x^t_{p_1p_1'}\cdots x^t_{p_np_n'}
x_{p_{n+1}p_{n+1}'}\cdots x_{p_Np_N'},
\end{equation}
where $x^t_{ij}=x_{ji}$ and
\begin{equation}
\gamma_N=
\left \{
  \begin{aligned}
   1    \quad \text{Case (O)},\\
    (-1)^n q^{2n}  \quad \text{Case (Sp)}.\\
  \end{aligned}
\right.
\end{equation}

\end{theorem}

\begin{proof}
For $i_1<i_2\cdots <i_m$, we can write

\begin{equation}
\begin{split}
X^{i_1,\cdots,i_m}_{i_1,\cdots,i_{m-1},j_m}
=\sum_{k=1}^{m}\check X^{i_1,\cdots,i_m}_{i_1,\cdots,i_{m-1},i_k}
x_{i_k,j_m}.\\
\end{split}
\end{equation}
Applying Proposition \ref{expansion auxiliary det},
we get that
\begin{align*}
X^{i_1,\cdots,i_m}_{i_1,\cdots,i_{m-1},j_m}&=\gamma_2 x_{i_{m-1},i_{m-1}}^t
X^{i_1, \cdots,i_{m-2}}_{i_1,\cdots, i_{m-2} }
x_{i_m, j_m }\\
&+\gamma_2\sum_{l=1}^{m-2}(-q)^{2l-2m+3}x^t_{i_l,i_{m-1}}
X^{i_1,\cdots,\hat{i_l},\cdots,i_{m-1}}_{i_1,\cdots,\hat{i_l},\cdots,i_{m-2},i_l}
x_{i_m, j_m }\\
&+\gamma_2\sum_{k=1}^{m-1}\sum_{l=1}^{k-1} (-q)^{2l-2k+2}
x_{i_l,i_k}^t
X^{i_1,\cdots,\hat{i_l},\cdots,\hat{i_k},\cdots,i_m}
_{i_1,\cdots,\hat{i_l},\cdots,\hat{i_k},\cdots,i_{m-1},i_l}
x_{i_k, j_m }\\
&+\gamma_2\sum_{k=1}^{m-1}\sum_{l=k+1}^{m-1} (-q)^{2l-2k}
x_{i_l,i_k}^t
X^{i_1,\cdots,\hat{i_k},\cdots,\hat{i_l},\cdots,i_m}
_{i_1,\cdots,\hat{i_k},\cdots,\hat{i_l},\cdots,i_{m-1},i_l}
x_{i_k, j_m }\\
&+\gamma_2\sum_{k=1}^{m-1} (-q)^{2m-2k-1}
x_{i_m,i_k}^t
X^{i_1,\cdots,\hat{i_k},\cdots,i_{m-1}}
_{i_1,\cdots,\hat{i_k},\cdots,i_{m-1}}
x_{i_k, j_m }
\end{align*}
Starting with $X^{1,\cdots,N}_{1,\cdots,N}$, we apply the recurrence relation repeatedly to write
the Sklyanin determinant $\sdet_{q}(X)$ in terms of the generator $x_{ij}$:
\begin{equation}
\sdet_{q} (X)=\gamma_N\sum_{p\in S_N}(-q)^{l(p)-l(p')}x^t_{p_1p_1'}\cdots x^t_{p_np_n'}
x_{p_{n+1}p_{n+1}'}\cdots x_{p_Np_N'}.
\end{equation}
The coefficient $\gamma_N$ is fixed by examining the leading term according to the two cases.
\end{proof}


The following theorem describes the center of $A_q(X_N)$ in the orthogonal case (c.f. Theorem \ref{sp center}).
\begin{theorem}
In the orthogonal case,
the center of algebra $A_q(X_N)$ is generated by $\sdet_q(X)$
and isomorphic to the polynomial ring in one variable.
\end{theorem}
\begin{proof}
It follows from Corollary \ref{sdet center} that
$\sdet_q(X)$ belongs to the center.
We introduce a total order among basic vectors
\begin{equation}
x^A=x_{11}^{a_{11}}x_{12}^{a_{12}}\cdots x_{1N}^{a_{1N}}
x_{22}^{a_{22}}x_{23}^{a_{23}}\cdots x_{2N}^{a_{2N}}
\cdots x_{N,N}^{a_{N,N}}, \qquad A\in Mat_N(\mathbb Z_+)
\end{equation}
of $A_q(X_N)$ by comparing the associated sequences
\begin{equation}
(\sum_{1\leq i<j\leq N}a_{ij},a_{11},a_{12},\dots ,a_{1N},a_{22},\dots,a_{N,N})\in \mathbb N^{N(N+1)/2+1}
\end{equation}
in the lexicographic order.
Let $p$ be the permutation
\begin{equation}
\begin{split}
(N-1,N-3,\cdots,4,2,1,3,\cdots,N)  \qquad  \text{ if $N$ is odd},\\
(N-1,N-3,\cdots,3,1,2,4,\cdots,N)  \qquad    \text{if $N$ is even}.
\end{split}
\end{equation}
Then the image $p'$ under the image of $\pi_N$ is $p$. So the leading term of $\sdet_q(X)$ is $x^{J}$ where $J$ is the identity matrix and the leading term of $(\sdet_q(X))^m$ is $x^{mJ}$.
Let $y$ be any element in the center of $A_q(X_N)$ with leading term
$cx^{A}$, $c\neq 0$. Then $yx_{ij}=x_{ij}y$ for any $1\leq i\leq j\leq N$. In particular, we consider
\begin{equation}
yx_{ii}\equiv q^{-2\sum_{k<i}a_{ki}} x^{A+e_{ii}}
\end{equation}
modulo lower terms. On the other hand,
\begin{equation}
x_{ii}y\equiv q^{-2\sum_{k>i}a_{ik}} x^{A+e_{ii}}
\end{equation}
modulo lower terms. Then we have that
\begin{equation}
\sum_{k<i}a_{ki}=\sum_{k>i}a_{ik}
\end{equation}
Taking $i=1$, we obtain that $\sum_{k>1}a_{ik}=0$.
It implies that $a_{1k}=0$ for $k>1$.

For $i=2,\cdots,N$, by repeating the same argument we obtain that
$a_{ij}=0$ for any $i<j$.

For $i<j$, we have that
\begin{equation}
\begin{split}
y x_{ij} \equiv q^{-2a_{jj}-\sum_{k>j}a_{ik}-\sum_{i<k<j}a_{kj}}
x^{A+e_{ij}}+\text{lower terms},
\\
x_{ij}y\equiv q^{-2a_{ii}-\sum_{k<i}a_{kj}-\sum_{i<k<j}a_{ik}}
x^{A+e_{ij}} +\text{lower terms}.
\end{split}
\end{equation}

Since $a_{ij}=0$ for $i<j$, we obtain that $a_{ii}={a_{jj}}$ for $i<j$.

Thus $y\equiv x^{mJ}$ for some $m$.
Let $y'=y-c(\sdet_q(X))^m$.
Then $y'$ also belongs to the center and its leading term is strictly lower than that of $y$.
By induction, we conclude
that $y'$ is a polynomial in $\sdet_q(X)$, so is $y$.
The powers of $\sdet_q(X)^m$  are linearly independent since they have linear independent leading terms.
Therefore, the center of $A_q(X_N)$ is isomorphic to the polynomial ring in one variable.
\end{proof}

\section{Minor identities for Sklyanin determinants}
In the following, we will work with the localizations of $A_q(X_N)$ by $\sdet_q(X)$ and $A_q(Mat_N)$
by $\det_q(T)$. In particular, we derive minor identities for the Sklyanin determinants.
Let's define the Sklyanin comatrix $\hat X$ by
by
\begin{equation}
\hat X X=\sdet_{q}(X)I,
\end{equation}
so the entries of $X^{-1}$ belong to $A_q(X_N)[\sdet_q(X)^{-1}]$.
\begin{proposition}\label{Skl comatrix}
The matrix elements  $\hat x_{ij}$ are given by
\begin{equation}
\hat x_{ij}=(-q)^{i-N}\check X^{1,\cdots,N}_{1,\cdot \hat i,\cdots ,N,j}
\end{equation}
Moreover,
\begin{equation}
\hat x_{ii}=X^{1,\cdot \hat i,\cdots ,N}_{1,\cdot \hat i,\cdots ,N}.
\end{equation}
\end{proposition}

\begin{proof}
Multiplying $X_N^{-1}$ from the right of the formulas
\begin{equation}
A_N \langle X_1,\dots,X_N \rangle=A_N \langle X_1,\dots,X_{N-1} \rangle
R^{t}_{1N}\cdots R^{t}_{N-1,N}X_N=A_N \sdet_q(X).
\end{equation}
we get that
\begin{equation}
A_N \langle X_1,\dots,X_{N-1} \rangle
R^{t}_{1N}\cdots R^{t}_{N-1,N}=A_N \hat X_{N}.
\end{equation}
Applying both sides to the vector
\begin{equation}
v_{ij}=e_1\otimes\cdots\hat{e_{i}}\otimes e_{N}\otimes e_j
\end{equation}
and comparing the coefficients of
$e_1\otimes\cdots \otimes e_{N} $ we get the first formula.
Using $R^t_{kN}v_{ii}=v_{ii}$ for $1\leq k\leq N-1$, applying the operators to the vector
$v_{ii}$ we obtain the second formula.
\end{proof}

 The matrix $X^{-1}=\sdet_q(X)^{-1}\hat X$  is neither a $q$-symmetric nor $q^{-1}$-symmetric (resp. antisymmetric) matrix in orthogonal (resp. symplectic) case.

Let $Q$ be the $N\times N$ diagonal matrix with $q_{ii}=(-q)^i$, $1\leq i\leq N$ and
$Y=Q^{-1}X^{-1} Q$. The following result shows that
$Y$ is a $q^{-1}$-symmetric (antisymmetric) matrix in the orthogonal (resp. symplectic) case and satisfies the $q^{-1}$-reflection relation.

\begin{proposition}
The matrix $Y$ satisfies the relation
\begin{equation}\label{q inverse reflection}
R^{-1}Y_1(R^{-1})^{t_1}Y_2=
Y_2 (R^{-1})^{t_1}Y_1 R^{-1}.
\end{equation}
and
\begin{align}
 &\mathrm{Case (O)}: \  y_{ij}=q^{-1}y_{ji}\quad (1\leq i<j\leq N), \\
 &\mathrm{Case (Sp)}:   y_{ii} =0 \, (1\leq i\leq N), y_{ji}=-q^{-1}y_{ij}\quad (1\leq i<j\leq N).
\end{align}

\end{proposition}

\begin{proof}
If follows from Eq. \eqref{RXRX} that
\begin{align}
R^{-1}X^{-1}_1(R^{t_1})^{-1}X^{-1}_2=
X^{-1}_2 (R^{t_1})^{-1}X^{-1}_1 R^{-1}.
\end{align}

Substituting $QYQ^{-1}$ for $X^{-1}$ we get that

\begin{align}
R^{-1}Q_1Y_1Q^{-1}_1(R^{t_1})^{-1}Q_2Y_2Q^{-1}_2=
Q_2 Y_2 Q^{-1}_2 (R^{t_1})^{-1}Q_1Y_1Q^{-1}_1 R^{-1}.
\end{align}

Multiplying $Q_1^{-1}Q_2^{-1}$ from the left and $Q_2Q_1$ from the right of both sides and noting
the relation
\begin{align}
R Q_1 Q_2=Q_2Q_1 R
\end{align}
we get that
\begin{align}
R^{-1}Y_1Q_2^{-1}Q^{-1}_1(R^{t_1})^{-1}Q_2Q_1Y_2=
Y_2 Q_1^{-1}Q^{-1}_2 (R^{t_1})^{-1}Q_1Q_2Y_1 R^{-1} .
\end{align}
Then relation \eqref{q inverse reflection} follows as $Q$ satisfies the equation:
$Q_1^{-1}Q^{-1}_2 (R^{t_1})^{-1}Q_1Q_2=(R^{-1})^{t_1}$.

Now let's consider the second part. Let $J(a)\in \mathrm{End}(\mathbb C^{N}\otimes \mathbb C^{N})$ be the matrix:
\begin{align}
 &\mathrm{Case (O)}: \  J(a)=J=\sum_{i=1}^{N}e_{ii},\\
 &\mathrm{Case (Sp)}:   J(a)=J=\sum_{i=1}^{n}(e_{2i-1,2i}-qe_{2i,2i-1}).
\end{align}

Regard $A_q(X_N)$ as subalgebra of $A_q(\mathrm{Mat}_N)$ by
\begin{equation}
X=TJT^t.
\end{equation}
Then $Y=Q^{-1}(T^t)^{-1}J^{-1}T^{-1} Q$. The matrix $T^{-1}$ satisfies the
relation
\begin{equation}
R^{-1}T^{-1}_1T^{-1}_2=T^{-1}_2T^{-1}_1R^{-1}.
\end{equation}
Denote the $ij$-th entry of $T^{-1}$ by $\hat t_{ij}$, then
$\hat t_{ij}={\det}_q (T)^{-1}(-q)^{i-j}\xi_{1,\cdots \hat i,\cdots,N}^{1,\cdots \hat j,\cdots,N}$
and $ij$-th entry of ${(T^{t})}^{-1}$ is given by
$det(T)^{-1}(-q)^{i-j}\xi_{1,\cdots \hat j,\cdots,N}^{1,\cdots \hat i,\cdots,N}$.
Therefore, we have that
\begin{equation}
(T^t)^{-1}=Q^2(T^{-1})^tQ^{-2}.
\end{equation}
The matrix $Y$ can be written as $Q(T^{-1})^tQ^{-2}J^{-1}T^{-1} Q$.

In the orthogonal case,
\begin{equation}
y_{ij}=\sum_{k=1}^N(-q)^{i+j-2k}\hat t_{ki}\hat t_{kj}
\end{equation}
It is easy to see that $y_{ij}=q^{-1}y_{ji}$ for $i<j$.

In the symplectic case,
\begin{equation}
y_{ij}=\sum_{k=1}^n(-q)^{i+j-4k}(q\hat t_{2k-1,i}\hat t_{2k,j}-\hat t_{2k,i}\hat t_{2k-1,j})
\end{equation}
Therefore $y_{ii}=0$ and $y_{ji}=-q^{-1}y_{ij}$ for $i<j$.

\end{proof}

\begin{proposition}
Let $A$ be the $N\times N$ antidiagonal matrix with $a_{i,N+1-i}=1$, $1\leq i\leq N$.
The map $X\mapsto AYA$ defines an algebra automorphism
$\omega$ of the localization of $A_q(X_N)$ by $\sdet_q(X)$.
\end{proposition}
\begin{proof}

Since the matrix $Y$ satisfies the $q^{-1}$-relations (\ref{zero diagonal}-\ref{relation noncommute}),
the matrix $AYA$ satisfies the $q$-relations.
Therefore, the map $X\mapsto AYA$ defines an algebra homomorphism.
Next, we show that $\omega$ is an involution.
Applying $\omega$ to the equation
\begin{equation}
XQA(AYA)A^{-1}Q^{-1}=I
\end{equation}
we get that
\begin{equation}
(AYA)QA\omega^2(X)A^{-1}Q^{-1}=I,
\end{equation}
which implies that
\begin{equation}
\omega^2(X)=(QA)^{-2}X(QA)^2.
\end{equation}
Since $(QA)^2=(-q)^{N+1}I$, we conclude that $\omega^2(X)=X$.

\end{proof}

The classical determinant holds significance in the realm of linear algebra and finds extensive applications in both mathematics and physics. Throughout its extensive history, numerous classical determinant identities have been discovered, often linked to prominent figures such as Cauchy, Jacobi, Muir, Sylvester, and others. For a review of these classical identities, the reader is referred to \cite{BS, bL}, where the second reference also gave quasideterminant analogs. In the subsequent discussion, we
extend these classical identities to encompass Sklyanin determinants and quantum Pfaffians.

The following theorem is the Sklyanin determinant analog of Jacobi's theorem.

\begin{theorem}\label{skly jacobi thm}
Let $I=\{i_1<i_2<\cdots<i_k\}$ be a subset of $[1,N]$ and
$I^{c}=\{i_{k+1}<\cdots<i_{N}\}$ the complement of $I$. Then
\begin{equation}
\sdet_{q^{-1}}(Y_{I^c})=\sdet_q(X_{I})\sdet_q(X)^{-1}.
\end{equation}
\end{theorem}

\begin{proof}
By the relation
\begin{equation}
A_{N} \langle X_1,\dots,X_{N} \rangle =\sdet_q(X) A_N
\end{equation}
and the definition of $\langle X_1, \cdots, X_N\rangle$ we have that
 \begin{equation}
 \begin{split}
&A_{N} \langle X_1,\dots,X_{k} \rangle
\overrightarrow{\prod}_{1\leq i\leq k<j\leq N}R_{ij}^t\\
&=
\sdet_q(X) A_{N}
X^{-1}_{N}(R_{N-1,N}^t)^{-1}X_{N-1}^{-1}\cdots  (R_{k+1,k+2}^t)^{-1}X_{k+1}^{-1}
\end{split}
\end{equation}

Since $RQ_1Q_2=Q_2Q_1R$ and $
Q_1^{-1}Q^{-1}_2 (R^{t_1})^{-1}Q_1Q_2=(R^{-1})^{t_1}
$
we get that
\begin{equation}
\begin{split}
&A_{N} \langle X_1,\dots,X_{k} \rangle
\overrightarrow{\prod}_{1\leq i\leq k<j\leq N}R_{ij}^t\\
&=
\sdet_q(X) A_{N} Q_{k+1}\dots Q_N
Y_{N}(R_{N-1,N}^{-1})^{t}Y_{N-1}\cdots (R_{k+1,k+2}^{-1})^{t}Y_{k+1}\\
&\qquad \qquad \cdot  Q_{k+1}^{-1}\dots Q_{N}^{-1}
\end{split}
\end{equation}
Applying both sides to the vector $e_{i_1}\otimes \dots e_{i_k}\otimes e_{i_N}\otimes\dots \otimes e_{i_{k+1}}$ and comparing the
coefficient of $e_{i_1}\otimes \dots e_{i_k}\otimes e_{i_N}\otimes\dots \otimes e_{i_{k+1}}$ we obtain that
\begin{equation}
\begin{split}
\sdet_q(X_{I})= \sdet_q(X)\sdet_{q^{-1}}(Y_{I^{c}}).
\end{split}
\end{equation}
\end{proof}

\begin{proposition} \label{jacobi comatrix}
 For any $1\leq a,b\leq k$, one has that
  \begin{equation}
X^{a,k+1,\cdots,N}_{b,k+1,\cdots,N} =(-q)^{k-b}\sdet_q(X) \check Y_{1,\cdots,\hat{a},\cdots,k,b}^{1,\cdots,k},
  \end{equation}
  where $Y=Q^{-1}X^{-1} Q$.
  \end{proposition}

  \begin{proof}
    The proof is similar to Jacobi's theorem.
    Using the relation
  \begin{equation}
  \begin{split}
    &A_{N} \langle X_1,\dots,X_{k+1} \rangle
    \overrightarrow{\prod}_{1\leq i\leq k,k+2\leq j\leq N}R_{ij}^t\\
  =&
  \sdet_q(X) A_{N} Q_{k+1}\cdots Q_N
  Y_{N}(R_{N-1,N}^{-1})^{t}\cdots Y_{k+2}\\
  &\cdot (R_{k+1,N}^{-1})^{t}\dots (R_{k+1,k+2}^{-1})^{t}  Q_{k+1}^{-1}\dots Q_{N}^{-1}
  \end{split}
  \end{equation}
  Applying both sides to the vector
  $e_{k+1}\otimes \dots e_{N}\otimes e_{b}\otimes e_k \otimes \dots \otimes \widehat {e_{a}}\dots \otimes e_{1}$
  and comparing the
  coefficient of $e_{1}\otimes e_{2}\otimes  \dots \otimes e_{N}$ we obtain that
  \begin{equation}
    X^{a,k+1,\cdots,N}_{b,k+1,\cdots,N} =(-q)^{k-b}\sdet_q(X) \check Y_{1,\cdots,\hat{a},\cdots,k,b}^{1,\cdots,k}
  \end{equation}
  \end{proof}

Using Jacobi's theorem we obtain the following analog of Cayley's complementary identity for the Sklyanin determinant.
\begin{theorem}\label{skly cayley thm }
Suppose a minor identity for the Sklyanin determinant is given:
\begin{equation}\label{skly identity}
\sum_{i=1}^{k}b_i \prod_{j=1}^{m_i}
\sdet_{q} (X_{I_{ij}})=0,
\end{equation}
where $I_{ij}'s$ are subsets of $[1,N]$  and $b_i\in \mathbb C(q)$.
Then the following identities hold
\begin{equation}
\sum_{i=1}^{k} b_i'\prod_{j=1}^{m_i}
\sdet_q(X)^{-1}
\sdet_{q} (X_{I_{ij}^{c}})=0,
\end{equation}
where $b_i'$ is obtained from $b_i$ by replacing $q$ by $q^{-1}$.
\end{theorem}

\begin{proof}
The matrix $Y$ satisfies the $q^{-1}$ relations. Applying the minor identity to $Y$ we get that
\begin{equation}
\sum_{i=1}^{k}b_i' \prod_{j=1}^{m_i}
\sdet_{q^{-1}} (Y_{I_{ij}})=0.
\end{equation}
It follows from Theorem \ref{skly jacobi thm} that
$\sdet_{q^{-1}} (Y_{I_{ij}})$ can be replaced by
$\sdet_{q}(X)^{-1}$ $\cdot\sdet_{q} (X_{I_{ij}^c})$.
This completes the proof.

\end{proof}

The following theorem is an analog of Muir's law for the Sklyanin determinant.
\begin{theorem}\label{skly muir law}
Suppose a minor Sklyanin determinant identity is given:
\begin{equation}
\sum_{i=1}^{k}b_i \prod_{j=1}^{m_i}
\sdet_{q} (X_{I_{ij}})=0,
\end{equation}
where $I_{ij}'s$ are subsets of $I=\{1,2,\dots,N\}$ and $b_i\in \mathbb C(q)$.
Let $J$ be the set $\{N,\dots, N+M\}$. Then the following identities hold
\begin{equation}
\sum_{i=1}^{k} b_i\prod_{j=1}^{m_i}
\sdet_{q} (X_{ J})^{-1}
\sdet_{q} (X_{I_{ij}\cup J})=0.
\end{equation}
\end{theorem}
\begin{proof}
Applying Cayley's complementary identity concerning the set $I$, we get that
\begin{equation}\label{obtain by jacobi thm}
\sum_{i=1}^{k} b_i'\prod_{j=1}^{m_i}
\sdet_q(X_{I})^{-1}
\sdet_{q} (X_{I\setminus I_{ij}})=0,
\end{equation}
Applying Cayley's
complementary identity for the set $I\cup J$,
we obtain that
\begin{equation}
\sum_{i=1}^{k} b_i\prod_{j=1}^{m_i}
\sdet_{q} (X_{ J})^{-1}
\sdet_{q} (X_{I_{ij}\cup J})=0.
\end{equation}

\end{proof}

Let  $A$ be a $N\times N$ matrix. For any subset
$I$ of $[1,N]$, we denote by $A_I$ the principal submatrix of $A$
with rows and columns
indexed by the elements
of $I$. The following relation involving the determinants
and the permanents were first established by Muir:
\begin{equation}
 \sum_{k=0}^N(-1)^{k}\sum_{I\subset [1,N]\atop |I|=k}\det(A_I)\per(A_{[1,N]\setminus I})=0.
\end{equation}

In the following, we give an analog of Muir's identity for the Sklyanin determinant.
Denote $(n)_q=\frac{q^n-q^{-n}}{q-q^{-1}}$.
Let $S_r$ be the $q$-symmetrizer:
\begin{equation}
S_{2}=\frac{1}{q^2-q^{-2}}\hat{R} (q),\   S_{m+1}=\frac{1}{q^{m+1}-q^{-m-1}}S_{m } \hat{R}_{m,m+1} (q^{m})S_{m }.
\end{equation}

\begin{theorem}
One has that
\begin{align}
\sum_{r=0}^{k}(-1)^r tr_{1,\ldots,k} S_{r}A_{k-r}' \langle X_1,\dots,X_k\rangle=0,\\
\sum_{r=0}^{k}(-1)^r tr_{1,\ldots,k}A_rS_{k-r}' \langle X_1,\dots,X_k\rangle=0,
\end{align}
where $A_{k-r}'$ and $S_{k-r}'$ denote the antisymmetrizer and symmetrizer over the copies of $End(\mathbb{C}^k)$ labeled by $\{r+1,\ldots,k\}$.
\end{theorem}

\begin{proof}

In the following, we show that
\begin{equation}\label{trace replacement}
\begin{aligned}
&tr_{1,\ldots,k}S_rA_{k-r}'\langle X_1,\dots,X_k\rangle\\
=&tr_{1,\ldots,k}  \frac{(r)_q(k-r+1)_q}{(k)_q}S_rA_{k-r+1}'\langle X_1,\dots,X_k\rangle\\
 &\quad +tr_{1,\ldots,k} S_rA_{k-r+1}'\frac{(k-r)_q(r+1)_q}{(k)_q}S_{r+1}A_{k-r}'
\langle X_1,\dots,X_k\rangle.
\end{aligned}
\end{equation}

The element $\langle X_1,\dots,X_k\rangle$ can be written as
\begin{equation}
 \langle X_1,\dots,X_r\rangle
\prod_{1\leq i\leq r \atop r+1\leq j\leq k}
R_{ij}^t
\langle X_{r+1},\dots,X_k\rangle ,
\end{equation}
where the product is taken in the lexicographical order on the pairs $(i,j)$.
It follows from \eqref{variant YBE}
that
\begin{align}
A_{k-r}'\prod_{1\leq i\leq r \atop r+1\leq j\leq k}
R_{ij}^t
&=\prod_{1\leq i\leq r \atop r+1\leq j\leq k}
R_{ij}^t
 A_{k-r}',\\
S_r\prod_{1\leq i\leq r \atop r+1\leq j\leq k}
R_{ij}^t
&=\prod_{1\leq i\leq r \atop r+1\leq j\leq k}
R_{ij}^t
S_r.
\end{align}
Then
\begin{align}
A_{k-r}'\langle X_1,\dots,X_k\rangle
&=\langle X_1,\dots,X_k\rangle
 A_{k-r}',\\
S_r\langle X_1,\dots,X_k\rangle
&=\langle X_1,\dots,X_k\rangle
S_r.
\end{align}
Thus we have that
\begin{equation}
\begin{aligned}
&tr_{1,\ldots,k}  S_rA_{k-r+1}'
 \langle X_1,\dots,X_k\rangle\\
&=\frac{1}{q^{k-r+1}-q^{r-k-1}}tr_{1,\ldots,k}
S_rA_{k-r}'\hat{R}_{r,r+1}(q^{r-k})A_{k-r}'\langle X_1,\dots,X_k\rangle \\
&=\frac{1}{q^{k-r+1}-q^{r-k-1}}tr_{1,\ldots,k}
S_r\hat{R}_{r,r+1}(q^{r-k})A_{k-r}'\langle X_1,\dots,X_k\rangle \\
\end{aligned}
\end{equation}
Similarly,
\begin{equation}
\begin{split}
&tr_{1,\ldots,k}  S_{r+1}A_{k-r}'
\langle X_1,\dots,X_k\rangle\\
&=\frac{1}{q^{r+1}-q^{-r-1}}tr_{1,\ldots,k}
S_r \hat{R}_{r,r+1}(q^{r})S_rA_{k-r}'\langle X_1,\dots,X_k\rangle \\
&=\frac{1}{q^{r+1}-q^{-r-1}}tr_{1,\ldots,k}
S_r\hat{R}_{r,r+1}(q^{r})A_{k-r}'\langle X_1,\dots,X_k\rangle.
\end{split}
\end{equation}
Using the equation $R^+-R^-=(q-q^{-1})P$, we have
\begin{equation}
\begin{split}
\frac{(r)_q}{(k)_q}
 \hat{R}_{r,r+1}(q^{r-k})
+\frac{(k-r)_q}{(k)_q}
 \hat{R}_{r,r+1}(q^{r})=q-q^{-1}
 \end{split}
\end{equation}
These imply the equation \eqref{trace replacement}. Therefore we have shown the first equation.
The second equation can be proved by the same arguments.
\end{proof}

The following is an analog of Sylvester's theorem for the Sklyanin determinant.
\begin{theorem}
 Let
 $I=\{1,\cdots,N\}$ ,
 $J=\{N+1,\cdots,N+M\}$, where $N$ and $M$ are positive integers such that $N$ and $M$ are even in the symplectic case.
Then the mapping $x_{ij}\mapsto  X^{i,M+1,\cdots,M+N}_{j,M+1,\cdots,M+N}$ defines an algebra morphism ${A_q(X_N)}\rightarrow {A_q(X_{N+M})}$. Denote $\tilde{x}_{ij}$ by the image of  $x_{ij}$
and $\tilde{X}=(\tilde x_{ij})$. Then
\begin{equation}
\sdet_q(\tilde{X})=\sdet_q(X_J)^{N-1}\sdet_q(X).
\end{equation}
\end{theorem}

\begin{proof}
Let $X$ be  the $(M+N)\times (M+N)$ matrix for ${A_q(X_{N+M})}$, we wirtie $Y=Q^{-1}X^{-1} Q$ as
\begin{equation}
  \begin{pmatrix}
      &Y_{11} &Y_{12}\\
      &Y_{21} &Y_{22}
    \end{pmatrix},
\end{equation}
where $Y_{11}$ is a $N\times N$ matrix and  $Y_{22}$ is a $M\times M$ matrix and
the inverse of $Y_{11}$ is  $(\sdet_{q^{-1}}{Y}_{11})^{-1}\hat{Y}_{11}$.
Let $Z=Q Y_{11}^{-1}Q^{-1}$, then $Z$ satisfies the $q$-reflection relation and the $q$-symmetry relation.
It follows from Proposition \ref{Skl comatrix} that the $(i,j)$-th entry of $\hat{Y}_{11}$
is $(-q)^{N-i}\check{Y}_{1,\cdots,\hat{i},\cdots,N,j}^{1,\cdots,N}$.
By Proposition \ref{jacobi comatrix} we have that
\begin{equation}
  \tilde x_{ij}=X^{i,M+1,\cdots,M+N}_{j,M+1,\cdots,M+N} = \sdet_q(X) \sdet_{q^{-1}}({Y}_{11} )  z_{ij}
    \end{equation}
Since $\sdet_q(X)$ and $\sdet_{q^{-1}}({Y}_{11} )$ commute with $z_{ij}$ for any $1\leq i,j\leq N$,
$\tilde X$ satisfies the $q$-reflection relation and the $q$-symmetry relation.
This proves the first statement.

By Jacobi's theorem,
\begin{equation}\label{X Y relation}
  \sdet_q(X) \sdet_{q^{-1}}({Y}_{11} )=\sdet_q(X_{J}) ,
    \end{equation}
then
$
  \tilde x_{ij}= \sdet_q(X_{J})  z_{ij}$.
Using the explicit formula for Sklyanin determinants,  we have that
\begin{equation}
 \sdet_q (\tilde X)= \sdet_q(X_{J})^{N}  \sdet_q (Z)
    \end{equation}
It follows from Jacobi's identity that
\begin{equation}
  \sdet_{q^{-1}}(Y_{11}) \sdet_{q}(Z)=1.
    \end{equation}
This implies that
\begin{equation}
 \sdet_{q}(Z)=\sdet_q(X_{J})^{-1} \sdet_q(X) .
\end{equation}
Therefore,
\begin{equation}
\sdet_q(\tilde{X})=\sdet_q(X_J)^{N-1}\sdet_q(X).
\end{equation}
\end{proof}

\section{Quantum Pfaffians}

A matrix $A$ is an $N\times N$ $q$-{\it antisymmetric} if
$a_{ii}=0$ and $a_{ji}=-qa_{ij}$, $i<j$.
The
quantum Pfaffian (or $q$-Pfaffian) of a $q$-antisymmetric matrix $A$ is defined by
\begin{align*}
\Pf_q(A)
=\frac{1}{(1+q^2)^n[n]_{q^4}!}\sum_{\sigma\in S_{2n}}(-q)^{l(\sigma)}a_{\sigma(1)\sigma(2)}a_{\sigma(3)\sigma(4)}\cdots a_{\sigma(2n-1)\sigma(2n)}.
\end{align*}

Let $I=\{i_1,i_2,\dots,i_{2r}\}$ be a subset of $[1,2n]$ with $i_1<i_2<\dots<i_{2r}$. Denote the complement of $I$ by $I^c$. Denote by $A_I$ the matrix obtained from $A$ by picking up the rows and columns indexed by $I$. We denote the quantum Pfaffian of $A_I$ by
$\Pf_q(A_I)=[i_1,i_2,\dots,i_{2r}]$.

Denote by $\Pi$ the set of $2$-shuffles, consisting of all $\sigma$ in $S_{2n}$ such that
$\sigma_{2k-1}<\sigma_{2k}$, $1 \leq k\leq n $ and $\sigma_{1}<\sigma_{3}<\cdots <\sigma_{2n-1}$.

\begin{proposition}\cite{JZ1}
If the $q$-antisymmetric matrix $A$ satisfies the condition:
\begin{equation}\label{plucker relation}
\begin{split}
&a_{ij}a_{kl}+(-q)a_{ik}a_{jl}+(-q)^{2}a_{il}a_{jk}\\
=&a_{kl}a_{ij}+(-q)^{-1}a_{jl}a_{ik}+(-q)^{-2}a_{jk}a_{il},
\end{split}
\end{equation}
where $i<j<k<l$, then
the quantum Pfaffian can be computed by
\begin{equation}\label{qlaplace}
\begin{split}
\Pf_q(A)&=\sum_{\pi \in\Pi}(-q)^{l(\pi)}[i_1,j_1][i_2,j_2]\cdots[i_n,j_n]\\
&=\sum_{j=2}^{2n}(-q)^{j-2}[1,j][2,3,\cdots,\hat{j},\cdots,2n].
\end{split}
\end{equation}
\end{proposition}

It is easy to verify that in the symplectic case the matrix $X=(x_{ij})_{1\leq i,j\leq N}$ is $q$-antisymmetric
and satisfies the condition
\begin{equation}
\begin{split}
&x_{ij}x_{kl}+(-q)x_{ik}x_{jl}+(-q)^{2}x_{il}x_{jk}\\
=&x_{kl}x_{ij}+(-q)^{-1}x_{jl}x_{ik}+(-q)^{-2}x_{jk}x_{il}
\end{split}
\end{equation}
for $i<j<k<l$. Thus for any sub-square matrix, these conditions are also met. Like in determinant,
we introduce the cofactor $X_{ij}$ by $X_{ii}=0$ and
\begin{equation}
\begin{split}
X_{ij}&=(-q)^{i-j}[1,\cdots,\hat{i},\cdots,\hat{j},\cdots,2n], i<j\\
X_{ij}&=(-q)^{i-j-1}[1,\cdots,\hat{j},\cdots,\hat{i},\cdots,2n], i>j.
\end{split}
\end{equation}

\begin{theorem}\label{cofacor of pfaffian}
The cofactors of the Pfaffian satisfy the orthogonality relations:
\begin{align}
\sum_{j=1}^{2n}[i,j]X_{jk}=\delta_{ik}\Pf_q(X),\\
\sum_{j=1}^{2n}X_{kj}[j,i]=\delta_{ik}\Pf_q(X).
\end{align}
\end{theorem}
\begin{proof} Both identities are shown by induction on $n$ similarly, so we just check the first one.
The case of $n=1$ is trivial.
Expanding $[2,3,\cdots,\hat{j},\cdots,2n]$ in the $q$-Laplace expansion \eqref{qlaplace} of
the Pfaffian, we have that for any fixed $k$
\begin{align*}
\Pf_q(X)&=(-q)^{k-2}[1,k][2,3,\cdots,\hat{k},\cdots,2n]\\
&+\sum_{1<i<j<k}
(-q)^{i+j-k-2}\left(
[1,i][k,j]-q[1,j][k,i]\right)\Pf_q(X_{\{1,i,j,k\}^c})\\
&+\sum_{1<i<k<j}
(-q)^{i+j-k-3}\left(
[1,i][k,j]-q[1,j][k,i]\right)\Pf_q(X_{\{1,i,k,j\}^c})\\
&+\sum_{1<k<i<j}
(-q)^{i+j-k-4}\left(
[1,i][k,j]-q[1,j][k,i]\right)\Pf_q(X_{\{1,k,i,j\}^c})
\end{align*}
where the sums are taken over all $i,j$ satisfying the corresponding conditions.
By relations (\ref{zero diagonal}-\ref{relation noncommute}), the factor in front of
$\Pf_q(X_{\{1,i,j,k\}^c})$ etc. can be expressed as follows.
\begin{align*}
&[1,i][k,j]-q[1,j][k,i]\\
&=\begin{cases}
[k,j][1,i]-q^{-1}[k,i][1,j]
-(q^2-q^{-2})
[k,1][i,j], & 1<i<j<k\\
[k,j][1,i]-(1-q^{-2})[k,1][i,j]-q^{-1}[k,i][1,j], & 1<i<k<j\\
[k,j][1,i]-q^{-1}[k,i][1,j], & 1<k<i<j\end{cases}.
\end{align*}
Denote the sum of all items with the first factor $[k,i]$ by $\alpha_{i}$. Then
\begin{align*}
\alpha_1&=(-q)^{k-3}[k,1][2,3,\cdots,\hat{k},\cdots,2n]\\
&-\sum_{1<i<j<k}
(-q)^{i+j-k-2}
(q^2-q^{-2})
[k,1][i,j]\Pf(X_{\{1,i,j,k\}^c})\\
&-\sum_{1<i<k<j}
(-q)^{i+j-k-3}(1-q^{-2})[k,1][i,j]
\Pf(X_{\{1,i,k,j\}^c})\\
&=(-q)^{k-3}[k,1][2,3,\cdots,\hat{k},\cdots,2n]\\
&+(q-q^{-1})[k,1]\sum_{i=2}^{k-1}
(-q)^{2i-k-2}\sum_{j\notin\{1,i,k\}}(-q)^{j-i+\beta_{j}}
[i,j]
\Pf(X_{\{1,i,k,j\}^c})
\end{align*}
where $\beta_j=0$ for $2\leq j\leq i-1$,
$\beta_j=-1$ for $i+1\leq j\leq k-1$,
and $\beta_j=-2$ for $k+1\leq j\leq 2n$.

By the induction hypothesis, the above can be simplified as follows. 
\begin{equation}
\begin{split}
\alpha_1&=
(-q)^{k-3}[k,1][2,3,\cdots,\hat{k},\cdots,2n]\\
&+(q-q^{-1})[k,1]\sum_{i=2}^{k-1}
(-q)^{2i-k-2}
\Pf(X_{\{1,k\}^c})\\
&=(-q)^{1-k}[k,1][2,3,\cdots,\hat{k},\cdots,2n]
\end{split}
\end{equation}
Similarly, we also have 
\begin{equation}
\alpha_i
=\begin{cases} (-q)^{i-k}[k,i][2,3,\cdots,\hat{k},\cdots,2n], & 2\leq i\leq k-1\\
(-q)^{i-k-1}[k,i][2,3,\cdots,\hat{k},\cdots,2n], & k+1\leq i\leq 2n\end{cases}.
\end{equation}
Therefore,
\begin{equation}\label{expansion of pf}
\sum_{j=1}^{2n}[k,j]X_{jk}=\Pf_q(X).
\end{equation}

If $i\neq  k$, $X_{jk}$ can be expanded as
\begin{equation}
X_{jk}=\sum_{l\notin\{i,j,k\}}(-q)^{a_{jl}}[i,l]\Pf_q(X_{\{i,j,k,l\}^c}),
\end{equation}
where $a_{jl}\in \mathbb Z$.
Then we have that
\begin{align}
\sum_{j=1}^{2n}[i,j]X_{jk}=\sum_{j=1}^{2n}\sum_{l\notin\{i,j,k\}}(-q)^{a_{jl}}[i,j][i,l]\Pf_q(X_{\{i,j,k,l\}^c}).
\end{align}
Note that $a_{lj}=a_{jl}+1$ for $j<l$,
therefore, $\sum_{j=1}^{2n}[i,j]X_{jk}=0$.
\end{proof}

\begin{theorem}\label{sp center}
The center of the algebra $A_q(X_N)$ is generated by $\Pf_q(X)$
and is isomorphic to the polynomial ring in one variable.
\end{theorem}
\begin{proof}
Let $X^*=(X_{ij})$, the Pfaffian analog of the adjoint matrix of $X$.
It follows from the
orthogonality relations that
\begin{equation}
\Pf_q(X)X=XX^*X=X\Pf_q(X),
\end{equation}
which implies that $\Pf_q(X)$ belongs to the center of $A_q(X_N)$.

We now order the monomials $x^A$ as follows. To each $A\in Mat_N(\mathbb Z_+)$
 we associate a sequence of integers
\begin{equation}
(\sum_{1\leq i<j\leq N}a_{ij},a_{12},a_{13},\dots ,a_{1N},a_{23},\dots,a_{N-1,N})\in \mathbb N^{N(N-1)/2+1}
\end{equation}
and order $x^{A}, A\in Mat_N(\mathbb Z_+)$ by the lexicographic order of these
sequences. This order gives rise to a total order among the basic vectors of $A_q(X_N)$:
\begin{equation}
x^A=x_{12}^{a_{12}}x_{13}^{a_{13}}\cdots x_{1N}^{a_{1N}}
x_{23}^{a_{23}}x_{24}^{a_{24}}\cdots x_{2N}^{a_{2N}}
\cdots x_{N-1,N}^{a_{N-1,N}}.
\end{equation}

The quantum Pfaffian $\Pf_q(X)$ has the leading term
$x^J$, where $J=\sum_{k=1}^{n}e_{2i-1,2i}$, subsequently the
leading term of $(\Pf_q(X))^m$ is $x^{mJ}$.
Let $y$ be any element in the center of $A_q(X_N)$ with the leading term
$cx^{A}$, $c\neq 0$. Then $yx_{ij}=x_{ij}y$ for any $1\leq i<j\leq N$. Now
\begin{equation}
yx_{ij}\equiv q^{-(\sum_{k>j}(a_{jk}+a_{ik})+\sum_{i<k<j}a_{kj})} x^{A+e_{ij}}
\end{equation}
modulo lower terms. Similarly 
\begin{equation}
x_{ij}y\equiv q^{-(\sum_{k<i}(a_{ki}+a_{kj})+\sum_{i<k<j}a_{ik})}x^{A+e_{ij}}
\end{equation}
modulo lower terms. Then we have that
\begin{equation}\label{e:exprel}
\sum_{k>j}(a_{jk}+a_{ik})+\sum_{i<k<j}a_{kj}=\sum_{k<i}(a_{ki}+a_{kj})+\sum_{i<k<j}a_{ik}
\end{equation}
for any $1\leq i<j\leq N$.

Taking $(i,j)=(1,2)$, relation \eqref{e:exprel} implies that
$a_{1k}=a_{1j}=0$ for $j\geq 3$.
Eventually, one gets that $a_{2k-1,j}=a_{2k,j}=0$ for $j\geq 2k+1$ by repeating this argument.
Taking $(i,j)=(2k-1,2k+1)$, we get that
$a_{2k+1,2k+2}=a_{2k-1,2k}$. Thus $y\equiv x^{mJ}$
for some $m$.
Let $y'=y-c(\Pf_q(X))^m$.
Then $y'$ also belongs to the center with the leading term strictly lower than that of $y$.
By induction concerning the order of the basis of $A_q(X_N)$, we conclude
that $y$ is a polynomial in the variable $\Pf_q(X)$.
Powers of Pfaffian $(\Pf_q)^m$ are linearly independent since they have linear independent leading terms.
Therefore, the center of $A_q(X_N)$ is isomorphic to the polynomial ring in one variable.
\end{proof}

Let $\Lambda_N$ be the quantum exterior algebra  $ \mathbb C \langle y_1, \cdots,y_N\rangle/I$, where $I$ is the ideal $(y_i^2, qy_iy_j+y_jy_i (i<j))$. For
simplicity we still use the same symbol $y_i$ for the quotient $y_i+I$.
We will simply write the element $x\otimes y$ as $xy$ or $yx$ for $x\in A_q(X_N)$ and $y\in  \Lambda_N$.

Let $\Omega=\sum_{1\leq i,j\leq N}x_{ij}y_{i}y_{j}\in A_q(X_N)\otimes \Lambda_N$, Then
\begin{equation}
\Omega^n=(1+q^2)^n[n]_{q^4}!\Pf_q(X)y_{1} y_2\cdots  y_{2n}.
\end{equation}

\begin{proposition} Under the homomorphic injection 
$\phi: A_q(X_N) \longrightarrow A_q(\mathrm{Mat}_N)$ in Theorem \ref{embedding thm} we have that
$$\phi(\Pf_q(X))=a_1a_2\dots a_n{\det}_q(T).$$
\end{proposition}
\begin{proof} Define the algebra homomorphism $\phi':A_q(X_N)\otimes \Lambda_N \longrightarrow A_q(\mathrm{Mat}_N)\otimes \Lambda_N$ by
$x\otimes y \mapsto \phi(x)\otimes y
$. Denote $Y=(y_1,\ldots,y_n)^t$.
Then $\Omega$ can be written as $Y^tXY$ and
$\phi'(\Omega)=(T^tY)^tJ(a)(T^tY)$.

Let $\omega_i=\sum_{j=1}^N t_{ji}\otimes y_j$.
Then
\begin{align*}
\omega_j\omega_i&=-q\omega_i\omega_j,\quad i<j, \\
\omega_i\omega_i&=0.
\end{align*}
As $T ^t X=(\omega_1,\ldots,\omega_{N})^{t}$, one has that
\begin{align*}
\phi'(\Omega)^n=(1+q^2)^n[n]_{q^4}!\Pf_q(J(a))\omega_{1} \omega_2\cdots  \omega_{N}\\
=(1+q^2)^n[n]_{q^4}!\Pf_q(J(a)){\det}_q(T)y_{1} y_2\cdots  y_{N}.
\end{align*}
Therefore, $\phi(\Pf_q(X))=\Pf_q(J(a)){\det}_q(T)=a_1a_2\dots a_n {\det}_q(T)$.
\end{proof}

\begin{theorem}
\label{sdet pf}
In the symplectic  case, the Sklyanin determinant $\sdet_q(X)$ is explicitly given by
\begin{equation}
\sdet_q (X)=
    q^{3n}\Pf_q(X)^2.
\end{equation}
This gives a formula for the square of Pfaffian:
\begin{align*}
&\Pf_q(X)^2=\\
&(-q)^{-n}\sum_{p\in S_N}(-q)^{l(p)-l(p')}x^t_{p_1p_1'}\cdots x^t_{p_np_n'}
x_{p_{n+1}p_{n+1}'}\cdots x_{p_Np_N'},
\end{align*}
where $x^t_{ij}=x_{ji}$.
\end{theorem}
\begin{proof} The result follows from $\Pf_q(X)=a_1a_2\cdots a_n\det_q T$, Proposition
\ref{sdet} and Theorem \ref{sdet formula}.
\end{proof}

\section{Minor identities for quantum Pfaffians}
For the quantum Pfaffian we have the following analog of Jacobi's theorem.

\begin{theorem}\label{jacobi thm}
Let $I=\{i_1<i_2<\cdots<2k\}$ be a subset of $[1,2n]$,
$I^{c}=\{i_{2k+1}<\cdots<i_{2n}\}$ be the complement of $I$. Then
\begin{equation}
\Pf_{q^{-1}}(Y_{I^c})=\Pf_{q}(X_{I})\Pf_{q}(X)^{-1}
\end{equation}
\end{theorem}

\begin{proof}

It follows from the Theorem \ref{skly jacobi thm} that
\begin{equation}
\begin{split}\Pf_{q^{-1}}(Y_{I^{c}})= \pm
 \Pf_q(X_{I}) \Pf_q(X)^{-1}
\end{split}
\end{equation}

Let $J$ be the matrix with entries ${J_{ij}},{1\leq i,j\leq 2n}$ such that
$J_{i_{2t-1}i_{2t}}=1,J_{i_{2t}i_{2t-1}}=-q$ all other entries $0$.
The mapping $X\mapsto J$ defines an algebra homomorphism. For matrix $J$ we have that
 \begin{equation}
 \begin{split}
\Pf_q(J)=\Pf_q(J_I)=\Pf_{q^{-1}}(J^{-1}_{I^c})=1.
\end{split}
\end{equation}
Therefore, $\Pf_{q^{-1}}(Y_{I^c})=\Pf_{q}(X_{I})\Pf_{q}(X)^{-1}$.

\end{proof}

Using the same arguments in the proofs of minor identities for
Sklyanin determinants we obtain Theorems \ref{pf cayley}-\ref{pf sylv thm} for the quantum Pfaffians.

\begin{theorem}[Cayley's complementary identity for quantum Pfaffians]\label{pf cayley}
Suppose a quantum minor Pfaffian identity is given:
\begin{equation}\label{pf identity}
\sum_{i=1}^{k}b_i \prod_{j=1}^{m_i}
\Pf_{q} (X_{I_{ij}})=0,
\end{equation}
where $I_{ij}'s$ are subsets of $[1,2n]$ with even  cardinality and $b_i\in \mathbb C(q)$.
Then the following identity holds
\begin{equation}
\sum_{i=1}^{k} b_i'\prod_{j=1}^{m_i}
\Pf_q(X)^{-1}
\Pf_{q} (X_{I_{ij}^{c}})=0,
\end{equation}
where $b_i'$ is obtained from $b_i$ by replacing $q$ by $q^{-1}$.
\end{theorem}

\begin{theorem}[Muir's law]
Suppose given a quantum minor Pfaffian identity
\begin{equation}
\sum_{i=1}^{k}b_i \prod_{j=1}^{m_i}
\Pf_{q} (X_{I_{ij}})=0,
\end{equation}
where $I_{ij}'s$ are subsets of $I=\{1,2,\dots,2n\}$ with even  cardinality and $b_i\in \mathbb C(q)$.
Let $J$ be the set $\{2n+1,\dots, 2n+2m\}$. Then the following identities holds
\begin{equation}
\sum_{i=1}^{k} b_i\prod_{j=1}^{m_i}
\Pf_{q} (X_{ J})^{-1}
\Pf_{q} (X_{I_{ij}\cup J})=0.
\end{equation}
\end{theorem}

\begin{theorem}[Sylvester-type Theorem]\label{pf sylv thm}
Let $I=\{1,2,\dots,2n\}$ and $J=\{2n+1,\cdots, 2n+2m\}$  the mapping $x_{ij}\mapsto Pf_q(X_{\{i,j\}\cup J})$ defines an algebra morphism $A_q(X_{2n})\rightarrow A_q(X_{2n+2m})$. Denote $\tilde{x}_{ij}$ by the image of  $x_{ij}$
and $\tilde{X}=(\tilde x_{ij})$. Then
\begin{equation}
\Pf_q(\tilde{X})=\Pf_q(X_J)^{m-1}\Pf_q(X_{I\cup J}).
\end{equation}
\end{theorem}

The following is an analogue of the Grassmann-Pl\"ucker relation for the quantum Pfaffian.

\begin{theorem}
Let $n$ and $m$ be odd numbers, and $I=\{1,2,\dots,n\}$, $J=\{n+1,2,\dots,n+m\}$. Then the following relation holds.
\begin{equation}\label{GP}
\begin{split}
\sum_{j=1}^n(-q)^{n-j}\Pf_q(X_{I\setminus \{j\}})\Pf_q(X_{\{j\}\cup J})\\
=\sum_{j=n+1}^{n+m}(-q)^{j-n} \Pf_q(X_{I\cup \{j\}})\Pf_q(X_{J\setminus \{j\}})
\end{split}
\end{equation}
\end{theorem}

\begin{proof}
The element $\Pf_q(X_{\{j\}\cup J})$ can be expanded as
\begin{equation}
\sum_{k=n+1}^{n+m}{(-q)}^{k-n-1}x_{jk}\Pf_q(X_{J\setminus\{k\}})
\end{equation}
The left-hand side of \eqref{GP} can be written as
\begin{equation}
\sum_{j=1}^n\sum_{k=n+1}^{n+m}(-q)^{k-j-1}\Pf_q(X_{I\setminus \{j\}})
x_{jk}\Pf_q(X_{J\setminus\{k\}})
\end{equation}

Similarly, we expand $\Pf_q(X_{I\cup \{j\}})$ on the right-hand side and get that
\begin{equation}
\sum_{j=n+1}^{n+m}\sum_{l=1}^{n}(-q)^{j-l-1} \Pf_q(X_{I\setminus \{l\}})x_{lj}\Pf_q(X_{J\setminus \{j\}}).
\end{equation}
This completes the proof.

\end{proof}

\section {Quasideterminant, $\sdet_q$ and $\Pf_q$}

The usual determinant can be written as a product of successive principal quasideterminants \cite{GR}. As the quasideterminants can be defined for
matrices over noncommutative ring, Krob and Leclerc \cite{KL} found that the quantum determinant $\det_q(T)$ also obeys this type of identity (cf. \cite{ER}).
In this subsection, we prove that similar identities hold for the Sklyanin determinants and quantum Pfaffians.

Let's recall the definition of quasideterminants \cite{GR}.
Suppose that matrix $X$ (with possibly noncommutative entries) is invertible, $Y=X^{-1}$,  and $y_{ji}$ (the $ji$-entry of Y) is invertible. The $ij$-th quasideterminant $|X|_{ij}$
is the following element:
$$|X|_{ij}={(y_{ji})}^{-1}.$$

For any $I\subset [1,N]$, let $X_{I}$ denote the submatrix with rows and columns indexed by $I$.
Then we have the following result.
\begin{theorem}
In the orthogonal case, the Sklyanin determinant can be expressed as the product of quasideterminants
\begin{equation}
\sdet_q(X)=x_{11}|X_{\{1,2\}}|_{22}\cdots|X_{\{1,\cdots,N\}}|_{NN}
\end{equation}
and the quasideterminants on the right-hand side commute with each other.
More generally,
for any permutation $\sigma\in S_{N}$,
\begin{equation}
\sdet_q(X)= x_{\sigma_1\sigma_1}
|X_{\{\sigma_1,\sigma_2\}}|_{\sigma_2,\sigma_2}\cdots
|X_{\{\sigma_1,\cdots,\sigma_N\}}|_{\sigma_N,\sigma_N}
\end{equation}
and the quasideterminants on the right-hand side commute with each other.
\end{theorem}
\begin{proof}
It follows from the generalized quantum Cramer relation $\hat X X=\sdet_q(X)I$ that
\begin{equation}
\hat X =\sdet_q(X)X^{-1}.
\end{equation}

Taking the $\sigma_N\sigma_N$-th entry, we get that
\begin{equation}
\hat X_{\sigma_N\sigma_N} =\sdet_q(X)(X^{-1})_{{\sigma_N\sigma_N}}.
\end{equation}
By Proposition \ref{Skl comatrix},

\begin{equation}
\hat X_{\sigma_N\sigma_N}=
X^{1,\cdots,\hat{\sigma_N},\cdots,N}_{1,\cdots,\hat{\sigma_N},\cdots,N}.
\end{equation}

Therefore,

\begin{equation}
\sdet_q(X)=
X^{1,\cdots,\hat{\sigma_N},\cdots,N}_{1,\cdots,\hat{\sigma_N},\cdots,N}
|X_{\{\sigma_1,\cdots,\sigma_N\}}|_{{\sigma_N\sigma_N}}.
\end{equation}
By induction on $N$ we obtain that
\begin{equation}
\sdet_q(X)= x_{\sigma_1\sigma_1}
|X_{\{\sigma_1,\sigma_2\}}|_{{\sigma_2\sigma_2}}\cdots
|X_{\{\sigma_1,\cdots,\sigma_N\}}|_{\sigma_N\sigma_N}.
\end{equation}
The quasideterminant $|X_{\{\sigma_1,\cdots,\sigma_N\}}|_{\sigma_N\sigma_N}$
commutes with all $x_{ij}$ with $i,j\neq \sigma_N$. Therefore it commutes with
$|X_{\{\sigma_1,\cdots,\sigma_k\}}|_{{\sigma_k\sigma_k}}$ for $1\leq k\leq N-1$.
By induction on $N$, the quasideterminants in the right-hand side commute with each other.
\end{proof}

For any permutation $\sigma\in S_{N}$ with $\sigma_{2k-1}<\sigma_{2k}$,
we define
$
\theta_{\sigma}(k)= \#\{i|1\leq i\leq 2k-2,\sigma_{2k-1}<\sigma_i<\sigma_{2k}\}$,
and $\theta_{\sigma}=\sum_{k=1}^{n}\theta_{\sigma}(k)$.
\begin{theorem}
In the symplectic case, the quantum Pfaffian can be expressed as a product of quasideterminants
\begin{equation}
\Pf_q(X)=x_{12}|X_{\{1,2,3,4\}}|_{34}\cdots|X_{\{1,\cdots,N\}}|_{N-1,N}
\end{equation}
and the quasideterminants in the right-hand side commute with each other.
More generally,
for any permutation $\sigma\in S_{N}$ with $\sigma_{2k-1}<\sigma_{2k}$
\begin{equation}
\Pf_q(X)=(-q)^{\theta_{\sigma}}x_{\sigma_1\sigma_2}
|X_{\{\sigma_1,\cdots,\sigma_4\}}|_{\sigma_3\sigma_4}\cdots
|X_{\{\sigma_1,\cdots,\sigma_N\}}|_{\sigma_{N-1}\sigma_{N}}
\end{equation}
and the quasideterminants in the right-hand side commute with each other.
\end{theorem}

\begin{proof} Recall the quantum Pfaffian orthogonality
\begin{equation}
 X^* X=\Pf_q(X)I,
\end{equation}
then we have
\begin{equation}
X^* =\Pf_q(X)X^{-1}.
\end{equation}

Taking the $\sigma_{N}\sigma_{N-1}$-th entry, we get that
\begin{equation}
 X_{\sigma_N\sigma_{N-1}}^* =\Pf_q(X)(X^{-1})_{{\sigma_N\sigma_{N-1}}}.
\end{equation}
By Theorem \ref{cofacor of pfaffian},

\begin{equation}
X^*_{\sigma_N\sigma_{N-1}}=
(-q)^{\sigma_{N}-\sigma_{N-1}-1 }
\Pf_q(X_{1,\cdots,\hat{\sigma_{N-1}},\cdots\hat{\sigma_N},\cdots,N}).
\end{equation}
Therefore,
\begin{equation}
\Pf_q(X)=(-q)^{\sigma_{N}-\sigma_{N-1}-1 }
\Pf_q(X_{1,\cdots,\hat{\sigma_{N-1}},\cdots\hat{\sigma_N},\cdots,N})
|X_{\{\sigma_1,\cdots,\sigma_N\}}|_{\sigma_{N-1} \sigma_N}.
\end{equation}
By induction on $N$ we obtain that
\begin{equation}
\Pf_q(X)=(-q)^{\theta_{\sigma}}x_{\sigma_1\sigma_2}
|X_{\{\sigma_1,\cdots,\sigma_4\}}|_{\sigma_3\sigma_4}\cdots
|X_{\{\sigma_1,\cdots,\sigma_N\}}|_{\sigma_{N-1}\sigma_{N}}.
\end{equation}

The quasideterminant $|X_{\{\sigma_1,\cdots,\sigma_N\}}|_{\sigma_{N-1}\sigma_{N}}$
 commute with all $x_{ij}$ with $i,j\notin\{\sigma_{N-1}\sigma_{N}\}$, then
 it commutes with $|X_{\{\sigma_1,\cdots,\sigma_{2k}\}}|_{\sigma_{2k-1}\sigma_{2k}}$
 for $\leq k\leq n-1$. By induction on $N$, the quasideterminants in the right-hand side commute with each other.
\end{proof}

\bigskip
\centerline{\bf Acknowledgments}
\medskip
The work is supported in part by the National Natural Science Foundation of China grant nos.
12171303 and 12001218, the Humboldt Foundation, the Simons Foundation grant no. MP-TSM-00002518,
and the Fundamental Research Funds for the Central Universities grant no. CCNU22QN002.

\bibliographystyle{amsalpha}

\end{document}